\documentclass[english]{amsproc}
\usepackage[english]{babel}

\usepackage{graphics}
\usepackage{amsfonts, amssymb, amscd, amsmath}
\usepackage{latexsym}
\usepackage[matrix,arrow,curve]{xy}
\usepackage{mathabx}
\usepackage{color}
\usepackage{tikz}
\usetikzlibrary{matrix}

\DeclareMathOperator{\lk}{lk} \DeclareMathOperator{\cone}{Cone}
 \DeclareMathOperator{\Ker}{Ker}

\DeclareMathOperator{\rk}{rk} \DeclareMathOperator{\sgn}{sgn}
\DeclareMathOperator{\im}{Im} 
 \DeclareMathOperator{\ver}{Vert}

 \DeclareMathOperator{\soc}{Soc}
\DeclareMathOperator{\Dh}{DH} \DeclareMathOperator{\WN}{WN}
\DeclareMathOperator{\vol}{Vol} \DeclareMathOperator{\Poly}{Poly}
\DeclareMathOperator{\Ann}{Ann} 
\DeclareMathOperator{\PD}{PD} \DeclareMathOperator{\covol}{covol}
\DeclareMathOperator{\const}{const} \DeclareMathOperator{\GL}{GL}
\DeclareMathOperator{\Flip}{Flip}
\DeclareMathOperator{\Cone}{Cone}
\DeclareMathOperator{\MultiFans}{MultiFans}
\DeclareMathOperator{\CharFunc}{CharFunc}
\DeclareMathOperator{\ch}{char} \DeclareMathOperator{\Mink}{Mink}
\DeclareMathOperator{\dep}{dep} \DeclareMathOperator{\Hilb}{Hilb}
 \DeclareMathOperator{\Hom}{Hom}
\DeclareMathOperator{\proj}{proj}

\newcommand{\Zo}{\mathbb{Z}}
\newcommand{\Ro}{\mathbb{R}}

\newcommand{\Zg}{\mathbb{Z}_{\geqslant 0}}
\newcommand{\Co}{\mathbb{C}}
\newcommand{\Qo}{\mathbb{Q}}

\newcommand{\ko}{\Bbbk}

\newcommand{\indn}{^{\langle n\rangle}}
\newcommand{\dm}{\overline{dm}}

\newcommand{\br}{\widetilde{\beta}}

\newcommand{\ca}[1]{\mathcal{#1}}
\newcommand{\inc}[2]{{[#1\!:\!#2]}}

\newcommand{\I}{\mathcal{I}}
\newcommand{\D}{\mathcal{D}}
\newcommand{\A}{\mathcal{A}}
\newcommand{\F}{\mathcal{F}}

\newcommand{\Matr}{Matr}
\newcommand{\GMP}{G\mathcal{MP}}

\newcommand{\Hr}{\widetilde{H}}
\newcommand{\dd}{\partial}

\newcounter{stmcounter}[section]

\numberwithin{equation}{section}

\theoremstyle{plain}
\newtheorem{cor}[stmcounter]{Corollary}

\newtheorem{thm}[stmcounter]{Theorem}
\newtheorem*{thmNo}{Theorem}
\newtheorem{conj}[stmcounter]{Conjecture}
\newtheorem{prop}[stmcounter]{Proposition}
\newtheorem{lemma}[stmcounter]{Lemma}
\newtheorem{problem}[stmcounter]{Problem}
\newtheorem{que}[stmcounter]{Question}
\newtheorem{claim}[stmcounter]{Claim}

\theoremstyle{definition}
\newtheorem{defin}[stmcounter]{Definition}

\theoremstyle{remark}
\newtheorem{ex}[stmcounter]{Example}
\newtheorem{rem}[stmcounter]{Remark}

\begin{document}

\title{Volume polynomials and duality algebras of multi-fans}

\author[A. Ayzenberg]{Anton Ayzenberg}
\address{Department of Mathematics, Osaka City University, Sumiyoshi-ku, Osaka 558-8585, Japan.}
\email{ayzenberga@gmail.com}

\author[M. Masuda]{Mikiya Masuda}
\address{Department of Mathematics, Osaka City University, Sumiyoshi-ku, Osaka 558-8585, Japan.}
\email{masuda@sci.osaka-cu.ac.jp}


\subjclass[2010]{Primary 52A39, 52B11, 05E45, 52C35; Secondary
05E40, 13H10, 52B05, 52B40, 52B70, 57N65, 55N91, 28A75, 51M25,
13A02 } \keywords{multi-fan, multi-polytope, volume polynomial,
Poincare duality algebra, Macaulay duality, Stanley--Reisner ring,
Minkowski theorem, Minkowski relations, cohomology of torus
manifolds}

\begin{abstract}
We introduce a theory of volume polynomials and corresponding
duality algebras of multi-fans. Any complete simplicial multi-fan
$\Delta$ determines a volume polynomial $V_\Delta$ whose values
are the volumes of multi-polytopes based on $\Delta$. This
homogeneous polynomial is further used to construct a Poincare
duality algebra $\A^*(\Delta)$. We study the structure and
properties of $V_\Delta$ and $\A^*(\Delta)$ and give applications
and connections to other subjects, such as Macaulay duality,
Novik--Swartz theory of face rings of simplicial manifolds,
generalizations of Minkowski's theorem on convex polytopes,
cohomology of torus manifolds, computations of volumes, and linear
relations on the powers of linear forms. In particular, we prove
that the analogue of the $g$-theorem does not hold for
multi-polytopes.
\end{abstract}

\maketitle

%
%
%
%
%
%

\tableofcontents

\section{Introduction}

There is a fundamental correspondence in algebraic geometry
\cite{Ful}:
\[
\left\{\parbox[c]{4cm}{\begin{center}Toric varieties \\of complex
dimension
$n$\end{center}}\right\}\leftrightsquigarrow\left\{\mbox{Rational
fans in $\Ro^n$}\right\}.
\]
One can read the information about toric variety from its fan.
Complete toric varieties correspond to complete fans, non-singular
varieties correspond to non-singular fans, and projective toric
varieties correspond to normal fans of convex polytopes.
Combinatorics of a fan and geometry of a toric variety are closely
connected. In particular, the rays of a fan correspond to the
divisors on toric variety and higher dimensional cones correspond
to the intersections of divisors.

In the work \cite{HM} Hattori and the second named author expanded
this setting to topological category and generalized the
above-mentioned correspondence in the following way:
\begin{equation}\label{eqMasudaHattoriCorresp}
\left\{\parbox[c]{3.5cm}{\begin{center}Torus manifolds\\ of real
dimension
$2n$\end{center}}\right\}\rightsquigarrow\left\{\mbox{Nonsingular
multi-fans in $\Ro^n$}\right\},
\end{equation}
which will be explained in a minute.

Let $X$ be a smooth closed oriented $2n$-manifold with an
effective action of an $n$-dimen\-si\-on\-al compact torus $T$ and
at least one fixed point. A closed, connected, codimension two
submanifold of $X$ will be called characteristic if it is a
connected component of the fixed point set of a certain circle
subgroup $S$ of $T$, and if it contains at least one $T$-fixed
point. The manifold $X$ together with a preferred orientation of
each characteristic submanifold is called a torus manifold.
Characteristic submanifolds are the analogues of divisors on a
toric variety.

Note, that there is no one-to-one correspondence in
\eqref{eqMasudaHattoriCorresp}: there may be different (in any
sense) torus manifolds producing the same multi-fan. Nevertheless,
multi-fans provide a convenient tool to study such manifolds.

A multi-fan is the central object of this paper. We recall the
precise definition later. Informally, a multi-fan is a collection
of cones in $V\cong\Ro^n$ with apex at the origin, coming with
multiplicities and satisfying certain geometrical restrictions.
Sometimes it is convenient to assume that there is a fixed lattice
$N\subset V$, and the rays of $\Delta$ are rational with respect
to $N$. The cones of a multi-fan may overlap nontrivially, which
makes a multi-fan more general and flexible object than an
ordinary fan, and provides many nontrivial examples.

A multi-polytope is defined as follows. Let $\Delta$ be a
simplicial multi-fan in $V\cong\Ro^n$. For each ray $l_i\in
\Delta$, we specify an affine hyperplane $H_i\subset V^*$
orthogonal to the linear span of $l_i$. A tuple
$P=(\Delta,H_1,\ldots,H_m)$ is called a simple multi-polytope
based on $\Delta$. The relation of the multi-polytope to the
multi-fan on which it is based, is exactly the same as the
relation of a polytope to its normal fan.

For any multi-polytope $P\subset V^*$ there is a function
$\Dh_P\colon V^*\setminus\bigcup H_i\to \Zo$ (the notation stands
for Duistermaat--Heckman, see \cite{HM}). Informally, for a
generic point $x\in V^*$ the value $\Dh_P(x)$ indicates how many
times the ``boundary'' of $P$ wraps around $x$. The precise
definition is given in Section \ref{secDefMultipoly}. For an
ordinary simple convex polytope this function takes value $1$
inside the polytope, and $0$ outside.

A multi-fan $\Delta$ is called complete if it satisfies certain
mild conditions (see \cite{HM} or Definition
\ref{defCompleteMultifan} below). For multi-polytopes based on
complete simplicial multi-fans, the function $\Dh_P$ is compactly
supported. We can define the volume of a multi-polytope $P$ as an
integral
\[
\vol(P):=\int_{V^*}\Dh_P d\mu
\]
(the measure $\mu$ is chosen such that the volume of a fundamental
domain of the dual lattice $N^*$ is $1$).

For a given simplicial multi-fan $\Delta$ consider the space
$\Poly(\Delta)$ of all multi-polytopes based on $\Delta$.
Following \cite{Tim} we call it the space of analogous
multi-polytopes. To specify an affine hyperplane orthogonal to a
line $\langle l_i\rangle\subset V$ one needs a single real number
$c_i$, the normalized distance from $H_i$ to the origin taken with
sign. This number is called the support parameter. Thus the space
$\Poly(\Delta)$ is isomorphic to $\Ro^m$, where $m$ is the number
of rays of $\Delta$. Support parameters $(c_1,\ldots,c_m)$ provide
the canonical coordinates on $\Poly(\Delta)$.

If $\Delta$ is complete, the volume gives a function on the space
of analogous polytopes: $\Poly(\Delta)\to \Ro$, $P\mapsto
\vol(P)$. Similarly to the case of actual convex polytopes,
studied by Pukhlikov--Khovanskii \cite{PKh} and Timorin
\cite{Tim}, this function is a homogeneous polynomial in the
support parameters.

\begin{thm}[\cite{HM}]\label{thmVolPoly}
Let $\Delta$ be a complete simplicial multi-fan in $\Ro^n$ with
$m$ rays. There exists a homogeneous polynomial $V_{\Delta}\in
\Ro[c_1,\ldots,c_m]$ of degree $n$ such that
$V_{\Delta}(c_1,\ldots,c_m)=\vol(P)$ for a multi-polytope $P\in
\Poly(\Delta)$ with support parameters $(c_1,\ldots,c_m)$.
\end{thm}

Following Timorin's approach \cite{Tim}, we proceed as follows.
Consider the ring $\D$ of differential operators with constant
coefficients, acting on $\Ro[c_1,\ldots,c_m]$. We have
$\D=\Ro[\dd_1,\ldots,\dd_m]$, where $\dd_i=\frac{\dd}{\dd c_i}$.
It is convenient to double the degree, so we assume that
$\deg\dd_i=2$. Given any nonzero homogeneous polynomial $\Psi\in
\Ro[c_1,\ldots,c_m]$ of degree $n$, consider the subspace
$\Ann(\Psi)\subset \D$, $\Ann(\Psi)=\{D\in \D\mid D\Psi=0\}$. It
is easily seen, that $\Ann(\Psi)$ is a graded ideal, and the
quotient algebra $\D/\Ann(\Psi)$ is finite-dimensional and
vanishes in degrees $>2n$. Moreover, $\D/\Ann(\Psi)$ is a
commutative Poincare duality algebra of formal dimension $2n$
\cite[Prop.2.5.1]{Tim}.

Now consider a complete simplicial multi-fan $\Delta$ and apply
this construction to the volume polynomial $V_{\Delta}$. In result
we obtain a Poincare duality algebra $\A^*(\Delta):=
\D/\Ann(V_{\Delta})$ associated with a multi-fan $\Delta$. The
main goal of this work is to study the volume polynomials and
investigate the structure of the corresponding algebras and to
show their relation to other topics in combinatorics, convex
geometry, commutative algebra, and topology.

The work has the following structure. In Sections
\ref{secDefMultifans} and \ref{secDefMultipoly} we review the
basic notions of the theory of multi-fans and in Section
\ref{secVolPolyIndexMap} we review the notion of the index map
which is the key ingredient in the construction of the volume
polynomial. In the work \cite{HM}, introducing multi-fans, the
existence of a lattice $N\cong\Zo^n\subset V$ was assumed, so that
multi-fans are non-singular (or at least rational) with respect to
this lattice. In our paper we consider general multi-fans,
probably non-rational. Instead of a lattice we assume that the
ambient space $V$ has a fixed inner product. This allows, in
particular, to define and compute volumes of multi-polytopes in
$V^*=V$ of dimensions smaller than $n$ (dealing with lattices,
only unimodular volumes make sense). The exposition of the
multi-fan theory is built to comply with this continuous setting.
Nevertheless, all statements in the introductory sections follow
from their lattice analogues discussed in \cite{HM}.

In Section \ref{secBasicPropVolPoly} we prove the basic
enumerative properties of the volume polynomial. While the values
of $V_\Delta$ are the volumes of multi-polytopes, the values of
its partial derivatives are the volumes of proper faces of these
multi-polytopes up to certain constants. These relations will be
used further in Section \ref{secMinkRels}.

In Section \ref{secFormulaVolPoly} we prove a general formula
(actually, a family of formulas) for the volume polynomial, and
indicate a geometrical procedure which allows to find non-trivial
linear identities on the powers of linear forms. For actual convex
polytopes our formula coincides with the Lawrence's formula
\cite{Law}, which is well known in computational geometry.

In Section \ref{secPDAmultifan} we review the general
correspondence between homogeneous polynomials and Poincare
duality algebras, known as the Macaulay duality. Using this
correspondence we obtain an algebra $\A^*(\Delta)$ as a Poincare
duality algebra corresponding to the volume polynomial $V_\Delta$.
One way to obtain this algebra is via differential operators as in
Timorin's approach. Another way involves the index map of a
multi-fan.

The structure of multi-fan algebras in some particular cases is
described in Section \ref{secStructMFanAlg}. Every (complete
simplicial) multi-fan has an underlying simplicial cycle. If this
cycle is a homology sphere $K$, then $\A^*(\Delta)$ is the
quotient of Stanley--Reisner algebra of $K$ by a linear system of
parameters, and the dimensions of its graded components are the
$h$-numbers of $K$. This is similar to ordinary fans. If the
underlying simplicial cycle is a homology manifold, the algebra
$\A^*(\Delta)$ is the quotient of the Stanley--Reisner algebra by
the linear system of parameters and by the certain ideal
introduced and studied by Novik--Swartz \cite{NS,NSgor}. In this
case the dimensions of the graded components of $\A^*(\Delta)$ are
the $h''$-numbers of $K$. A short exposition of the Novik--Swartz
theory is provided.

Section \ref{secMinkRels} aims to generalize a classical Minkowski
theorem on convex polytopes to multi-polytopes. The direct
Minkowski theorem has a straightforward generalization which can
be used to obtain linear relations in the algebra $\A^*(\Delta)$.
On the other hand, the inverse Minkowski theorem, properly
formulated, is controlled by the power map
$\A^2(\Delta)\to\A^{2n-2}(\Delta)$, $a\mapsto a^{n-1}$.

In Section \ref{secRecognizing} we answer the question which
polynomials are volume polynomials of multi-fans, and which
Poincare duality algebras are algebras of multi-fans. We prove
that every Poincare duality algebra generated in degree $2$ is
isomorphic to $\A^*(\Delta)$ for some complete simplicial
multi-fan $\Delta$.

The basic operations on multi-fans, such as flips and connected
sums, and their effects to multi-fan algebras are described in
Section \ref{secSurgery}. In particular, we prove that, under
flips, the dimensions of graded components of $\A^*(\Delta)$
change similarly to $h$-numbers of simplicial complexes.

Finally, in Section \ref{secTorusManif} we discuss the relation of
$\A^*(\Delta)$ to the cohomology of torus manifolds. It is known
that, for complete smooth toric variety $X$, the cohomology ring
$H^*(X;\Ro)$ coincides with the algebra  $\A^*(\Delta_X)$ of the
corresponding fan. Situation with general torus manifolds and
their multi-fans is more complicated. Nevertheless, in a certain
sense, the multi-fan algebra $\A^*(\Delta_X)$ gives a ``lower
bound'' for the cohomology of a torus manifold.

%
%
%
%
%
%

\section{Definitions: multi-fans}\label{secDefMultifans}

%
%

\subsection{Multi-fans as parametrized collections of
cones}\label{subsecMFansBasics}

Let us recall the definition and basic properties of multi-fans.
This exposition follows the lines of~\cite{HM}.

Consider an oriented vector space $V\cong \Ro^n$ with a lattice
$N\subset V$, $N\cong \Zo^n$. A subset of the form
$\kappa=\{r_1v_1+\cdots+r_kv_k\mid r_i\geqslant 0\}$ for given
$v_1,\ldots,v_k\in V$ is called a cone in $V$. Dimension of
$\kappa$ is the dimension of the linear hull of $\kappa$. A cone
is called strongly convex if it contains no line through the
origin. In the following all cones are assumed strongly convex.

Using classical construction of supporting hyperplane one can
define the faces of $\kappa$, which are also the cones of smaller
dimensions. If the generating set $v_1,\ldots,v_k$ may be chosen
linearly independent (resp. rational, the part of basis of the
lattice $N$), $\kappa$ is called simplicial (resp. rational,
unimodular). Let $\Cone(V)$ denote the set of all cones in $V$.
This set obtains a partial order: $\kappa_1\prec \kappa_2$
whenever $\kappa_1$ is a face of $\kappa_2$.

Let $\Sigma$ be a finite partially ordered set with the minimal
element $*$. Suppose there is a map $C\colon\Sigma\to\Cone(V)$
such that

\begin{enumerate}
\item $C(*)=\{0\}$;

\item If $I<J$ for $I,J\in S$, then $C(I)\prec C(J)$;

\item For any $J\in \Sigma$ the map $C$ restricted on $\{I\in S\mid I\leqslant J\}$
is an isomorphism of ordered sets onto $\{\kappa\in \Cone(V)\mid
\kappa\preceq C(J)\}$.
\end{enumerate}

The image $C(\Sigma)$ is a finite set of cones in $V$. We may
think of a pair $(\Sigma, C)$ as a set of cones in $V$ labeled by
the ordered set $\Sigma$.

The poset $\Sigma$ obtains a rank function: $\rk(I):= \dim C(I)$.
The set of elements in $\Sigma$ having maximal rank $n$ is denoted
$\Sigma\indn$.


Consider an arbitrary function $\sigma\colon \Sigma\indn\to
\{-1,+1\}$ called a sign function.

\begin{defin}[Old definition]\label{defMultifanOld}
The triple $\Delta:=(\Sigma,C,\sigma)$ is called \emph{a
multi-fan} in $V$. The number $n=\dim V$ is called the dimension
of $\Delta$.
\end{defin}

Multi-fan $\Delta$ is called simplicial (resp. rational,
non-singular) if the values of $C$ are simplicial (resp. rational,
unimodular) cones. In the following we will always assume that
$\Delta$ is simplicial. Then every cone of $\Delta$ is simplicial
and property (3) of the map $C$ implies that $\Sigma$ is a
simplicial poset. Recall that a poset $\Sigma$ is called
simplicial if any lower order ideal $S_{\leqslant J}:=\{I\in S\mid
I\leqslant J\}$ is isomorphic to the poset of faces of a simplex
(i.e. a boolean lattice).

%
%

\subsection{Multi-fans as pairs of weight and characteristic
functions}\label{subsecMFansWeightCharFunc}

Note that definition \ref{defMultifanOld} of a multi-fan slightly
differs from the definition of multi-fan given in \cite{HM}. To
establish the correspondence consider the following construction.
Let $[m]=\{1,\ldots,m\}$ denote the set of vertices of $\Sigma$.

%

The signs of maximal simplices in $\Sigma$ determine two functions
on ${[m]\choose n}$, the set of all $n$-subsets of $[m]$:
\[
w^\pm\colon {[m]\choose n}\to \Zg,
\]
where $w^+(\{i_1,\ldots,i_n\})$ (resp. $w^-(\{i_1,\ldots,i_n\})$)
equals the number of simplices $I\in \Sigma\indn$ on the vertices
$\{i_1,\ldots,i_n\}$ having sign $+1$ (resp. $-1$). Although both
functions $w^+,w^-$ are important by topological reasons (see
\cite{HM}), only their difference $w:= w^+-w^-$ is relevant to our
work. So far $w$ is a function which assigns an integral number to
each $n$-subset of $[m]$. Let us consider a pure simplicial
complex $K$ on the set $[m]$ whose maximal simplices $K\indn$ are
the subsets $I\subset [m]$ satisfying $w(I)\neq 0$. To reach
greater generality we allow $w$ to take real values, thus
\[
w\colon {[m]\choose n}\to \Ro.
\]

Each vertex $i\in[m]$ corresponds to a ray (i.e. $1$-dimensional
cone) of $\Delta$. We choose a generator in each ray. This gives a
so called characteristic map $\lambda\colon [m]\to V$, such that
the ray $C(i)$ is generated by $\lambda(i)$ for every $i\in [m]$.
It satisfies the following property:
\[
\mbox{if } \{i_1,\ldots,i_k\}\in K, \mbox{ then }
\lambda(i_1),\ldots,\lambda(i_k)\in V  \mbox{ are linearly
independent.}
\]
This condition is called \emph{$*$-condition}.

Note that in \cite{HM} all multi-fans were assumed rational. In
this case the generator $\lambda(i)$ can be chosen canonically as
a unique primitive integral vector contained in $C(i)$. Since we
want to include non-rational simplicial multi-fans in our
consideration, we should specify the generators somehow in order
for the subsequent calculations to make sense.

Finally we get to the following definition

\begin{defin}[New definition]\label{defMultifanNew}
A triple $(K, w,\lambda)$ is called \emph{a simplicial multi-fan}
in $V$. Here $w\colon {[m]\choose n}\to \Ro$ is \emph{a weight
function}, $K$ is a simplicial complex which is the support of
$w$, and $\lambda\colon [m]\to V$ is \emph{a characteristic
function}. Characteristic function satisfies $*$-condition with
respect to $K$: if $I=\{i_1,\ldots,i_k\}\in K$, then the vectors
$\lambda(i_1),\ldots,\lambda(i_k)$ are linearly independent in
$V$.
\end{defin}

Here $K$ may have ghost vertices, i.e. $i\in[m]$ such that
$\{i\}\notin K$. The value of characteristic function in such
vertices may be arbitrary (even zero). In the following we will
not pay too much attention to ghost vertices since their presence
does not affect the calculations.

Strictly speaking, the new definition is not equivalent to the old
one, since we cannot restore the poset $\Sigma$ and the sign
function $\sigma\colon\Sigma\indn\to \{\pm1\}$ when $w$ takes
non-integral values. Even in the integral case we cannot restore
$\Sigma$ uniquely. On the other hand, as was shown above, every
multi-fan in the sense of old definition determines a multi-fan in
the sense of new definition. We will work with the new definition
most of the time.

\begin{rem}\label{remBasSituations}
When passing from the old definition to the new one, we may lose
an important information. For example consider the multi-fan in
$\Ro^2=\langle e_1,e_2\rangle$ whose maximal cones are two copies
of the non-negative cone (i.e. the cone generated by basis vectors
$e_1, e_2$), and two rays are generated by $e_1$ and $e_2$. One of
the maximal cones is taken with the sign $+1$ and the other with
the sign $-1$. We remark that such multi-fan corresponds to the
torus manifold $S^4$ \cite{HM}. We have
$w^+(\{1,2\})=w^-(\{1,2\})=1$, therefore $w(\{1,2\})=0$. Thus $K$
is empty (equivalently, $w\colon {[m]\choose n}\to \Ro$ vanishes).

One way to avoid such situations is to assume in the beginning
that $\Sigma$ itself is a simplicial complex rather than a general
simplicial poset. In this case $K$ coincides with $\Sigma$ and the
weight function $w$ on $K$ coincides with the sign function
$\sigma$. In particular, $w$ takes the value $\pm1$ on each
maximal simplex of $K$ (see Example \ref{exMFfromSimpComp}).
\end{rem}

%
%

\subsection{Underlying simplicial
chain}\label{subsecMFansSimpChain}

Let $\triangle_{[m]}$ denote an abstract simplex on the vertex set
$[m]$, and let $\triangle_{[m]}^{(n-1)}$ be its $(n-1)$-skeleton.
Every subset $I\subset [m]$, $|I|=n$ may be considered as a
maximal simplex of $\triangle_{[m]}^{(n-1)}$. If $I\in K\indn$,
then we can orient $I$ as follows: we say that the order of
vertices $(i_1,\ldots,i_n)$ of $I$ is positive if and only if the
basis $(\lambda(i_1),\ldots,\lambda(i_n))$ determines the positive
orientation of $V$.

\begin{defin}
The element
\[
w_{ch}=\sum_{I\subset K\indn} w(I)I\in C_{n-1}(K;\Ro)\subseteq
C_{n-1}(\triangle_{[m]}^{(n-1)};\Ro)
\]
is called \emph{the underlying chain of a multi-fan} $\Delta$.
Here $C_{n-1}(K;\Ro)$ denotes the group of simplicial chains of
$K$.
\end{defin}

%
%

\subsection{Complete multi-fans}\label{subsecCompleteMFans}

Let us briefly recall the notion of projected multi-fan. We give
the construction in terms of new definition of multi-fan although
the similar construction may be given in terms of simplicial
posets and sign functions.

Let $\Delta=(K,w,\lambda)$ be a simplicial multi-fan in the space
$V$, and let $I=\{i_1,\ldots,i_k\}\in K$ be a simplex. Let $V_I$
denote the quotient vector space
$V/\langle\lambda(i_1),\ldots,\lambda(i_k)\rangle$. Consider the
multi-fan $\Delta_I=(\lk_KI,w_I,\lambda_I)$ in $V_I$ defined as
follows:
\begin{itemize}
\item $\lk_KI:=\{J\subset [m]\setminus I\mid I\cup J\in K\}$ is
the link of the simplex $I$ in $K$.
%
\item $w_I(J):= w(I\cup J)$ for every $J\in \lk_KI$,
$|J|=n-|I|$.
\item $\lambda_I(j)$ is the image of $\lambda(j)\in V$ under the
natural projection $V\to V_I=
V/\langle\lambda(i_1),\ldots,\lambda(i_k)\rangle$. It is easily
seen that $\lambda_I$ satisfies $*$-condition.
\end{itemize}
If we choose some orientation of a simplex $I\in K$, the space
$V_I$ obtains an orientation induced from $V$. To be precise, let
us say that the basis $([v_1],\ldots,[v_{n-k}])$ determines a
positive orientation of $V_I$ if the basis
$(v_1,\ldots,v_{n-k},\lambda(i_1),\ldots,\lambda(i_k))$ is a
positive basis of $V$ for a chosen positive order
$(i_1,\ldots,i_k)$ of vertices of $I$.

We call $\Delta_I$ \emph{the projected multi-fa}n of $\Delta$. The
construction satisfies the hereditary relation
$(\Delta_{I_1})_{I_2}=\Delta_{I_1\sqcup I_2}$ whenever it makes
sense, and there holds $\Delta_{\varnothing}=\Delta$.

Let us call a vector $v\in V$ generic with respect to $\Delta$ if
it is not contained in the vector subspaces spanned by the cones
of $\Delta$ of dimensions $<n$. For any such $v$ define the number
$d_v=\sum w(I)\in \Ro$, where the sum is taken over all subsets
$I=\{i_1,\ldots,i_n\}\subset[m]$ such that the cone generated by
$\lambda(i_1),\ldots,\lambda(i_n)$ contains $v$.

\begin{defin}\label{defCompleteMultifan}
The multi-fan $\Delta$ is called \emph{pre-complete} if $d_v$ does
not depend on a generic vector $v\in V$. In this case $d_v$ is
called \emph{the degree} of $\Delta$. The multi-fan $\Delta$ is
called \emph{complete} if the projected multi-fan $\Delta_I$ is
pre-complete for any simplex $I\in K$.
\end{defin}

\begin{rem}
Note that this definition allows $w$ to be constantly zero. We
call a multi-fan zero if its weight function constantly zero. A
zero multi-fan is pre-complete and therefore complete.
\end{rem}

\begin{prop}\label{propCompleteCriterion}
A multi-fan $\Delta$ is complete if and only if its underlying
simplicial chain $w_{ch}\in C_{n-1}(\triangle_{[m]}^{(n-1)};\Ro)$
is a cycle, that is $dw_{ch}=0$ for the standard simplicial
differential $d\colon C_{n-1}(\triangle_{[m]}^{(n-1)};\Ro)\to
C_{n-2}(\triangle_{[m]}^{(n-1)};\Ro)$ (if $n=1$, we assume that
$d\colon C_0(\triangle_{[m]}^{(n-1)};\Ro)\to \Ro$ is the
augmentation map).
\end{prop}

\begin{proof}
In the case when $w$ takes only integral values, the statement is
proved in \cite[Sec.6]{HM}. If $w$ takes only rational values,
scaling the values of $w$ by a common denominator reduces the task
to the integral case. It remains to prove the statement for
real-valued $w$. Both conditions ``$\Delta$ is complete'' and
``$dw_{ch}=0$'' determine rational vector subspaces in the space
of all possible weight functions (it is not difficult to define
the pre-completeness condition in terms of the ``wall-crossing
relations'', which are linear relations on $w(I)$ with integral
coefficients). Thus the rational case implies the real case.
\end{proof}

For convenience we summarize the discussion by the following
definition.

\begin{defin}[Complete simplicial
multi-fan]\label{defCompleteMfanCycle} \emph{A complete simplicial
multi-fan} is a pair $(w_{ch},\lambda)$, where
$w_{ch}=\sum_{I\subset [m], |I|=n}w(I)I\in
Z_{n-1}(\triangle_{[m]}^{(n-1)})$ is a simplicial cycle on $m$
vertices, and $\lambda\colon [m]\to V$ is any function satisfying
the condition: $\{\lambda(i)\}_{i\in I}$ is a basis of $V$ if
$|I|=n$ and $w(I)\neq 0$.
\end{defin}

For a complete multi-fan $\Delta$ the corresponding homology class
$[w_{ch}]\in \Hr_{n-1}(K;\Ro)\subset
\Hr_{n-1}(\triangle_{[m]}^{(n-1)};\Ro)$ will be denoted $[\Delta]$
and called the underlying homology class of~$\Delta$. Since
$C_n(\triangle_{[m]}^{(n-1)};\Ro)=0$, the groups
$Z_{n-1}(\triangle_{[m]}^{(n-1)})$ and
$\Hr_{n-1}(\triangle_{[m]}^{(n-1)})$ may be identified. Thus
$w_{ch}$ and $[\Delta]$ are just two different notations for the
same object.

\begin{ex}\label{exMFfromSimpComp}
One obvious way to obtain a complete multi-fan is to start with
any oriented pseudomanifold $K$ of dimension $n-1$ on the set of
vertices $[m]$, and take any characteristic function
$\lambda\colon[m]\to V$. Since $K$ is oriented, every maximal
simplex $I$ of $K$ becomes oriented, but this orientation may be
different from the one determined by characteristic function (see
subsection \ref{subsecMFansSimpChain}). Let $w(I)$ be $+1$ or $-1$
depending on whether these two orientations agree or not. Let us
extend the weight function by zeroes to non-simplices of $K$. The
corresponding simplicial chain $w_{ch}=\sum_Iw(I)I\in
C_{n-1}(\triangle_{[m]}^{(n-1)};\Ro)$ is closed, since it is
exactly the fundamental chain of $K$ in $\triangle_{[m]}^{(n-1)}$.
Therefore, $(w_{ch},\lambda)$ is a complete simplicial fan.
\end{ex}

\begin{ex}
The previous example may be restricted to the case when $K$ is a
homology sphere or homology manifold. We will study these two
cases in more detail in Section \ref{secStructMFanAlg}.
\end{ex}

We say that $\Delta$ is based on an orientable simplicial
pseudomanifold $K$ if the corresponding simplicial cycle is given
by $K$.

There is one interesting feature of (complete) multi-fans revealed
by Definitions \ref{defMultifanNew} and
\ref{defCompleteMfanCycle}. The multi-fans with the given set of
vertices $[m]$ and the given characteristic function $\lambda$
form a vector space: we may add them by adding their weights and
multiply by real numbers by scaling their weights. Let
$\MultiFans_\lambda$ denote the vector space of complete
multi-fans with the given characteristic function~$\lambda$. This
space may be identified with certain vector subspace of
$Z_{n-1}(\triangle_{[m]}^{(n-1)};\Ro)$. We will discuss this
subspace in detail in subsection
\ref{subsecCharVolPolyNonGeneral}. The set of multi-fans with
integral weights forms a lattice inside $\MultiFans_\lambda$ which
is a certain sublattice of $Z_{n-1}(\triangle_{[m]}^{(n-1)};\Zo)$.

%
%
%
%
%
%
%

\section{Definitions: multi-polytopes}\label{secDefMultipoly}

%
%

\subsection{Multi-polytopes}

Let $\Delta$ be a simplicial multi-fan with characteristic
function $\lambda\colon[m]\to V$. Let $HP(V^*)$ denote the set of
all affine hyperplanes in the dual vector space $V^*$.

For each $i\in [m]$ choose an affine hyperplane $\ca{H}(i)\subset
V^*$ in the dual space which is orthogonal to the linear hull of
the $i$-th cone. In other words, $\ca{H}(i)$ is defined by
equation $\ca{H}(i)=\{u\in V^*\mid\langle u,
\lambda(i)\rangle=c_i\}$ for some constant $c_i\in \Ro$ called
\emph{the support parameter} of $\ca{H}(i)$.

\begin{defin}\label{defMultiPoly}
\emph{A multi-polytope} $P$ is a pair $(\Delta,\ca{H})$, where
$\Delta$ is a multi-fan, and $\ca{H}\colon [m]\to HP(V^*)$ is a
function such that $\ca{H}(i)$ is orthogonal to $\lambda(i)$ for
any $i\in[m]$. We say that $P$ is based on the multi-fan $\Delta$.
\end{defin}

Although the definition may be stated in general, we restrict to
simplicial multi-fans $\Delta$, in which case $P$ is called a
simple multi-polytope.

Let us denote the set of all multi-polytopes based on $\Delta$ by
$\Poly(\Delta)$. Every such multi-polytope is completely
determined by its support parameters $c_1,\ldots,c_m$. Thus
$\Poly(\Delta)$ has natural coordinates $(c_1,\ldots,c_m)$ and may
be identified with $\Ro^m$. This space is called \emph{the space
of analogous polytopes} based on $\Delta$.

To simplify notation, we denote $\ca{H}(i)$ by $H_i$ and set
\[
H_I := \bigcap\nolimits_{i\in I}H_i \mbox{ for } I\in K.
\]
$H_I$ is a codimension $|I|$ affine subspace in $V^*$, since the
normals of the hyperplanes $H_i$, $i\in I$ are linearly
independent by $*$-condition. In particular, when $I$ is a maximal
simplex, $I\in K\indn$, $H_I$ is a point in $V^*$ which is called
the vertex of $P$.

\begin{defin}
Let $\Delta$ be a simplicial multi-fan in $V$ with the underlying
simplicial complex $K$ and let $P$ be a simple multi-polytope
based on $\Delta$. Let $I\in K$. Consider a simple multi-polytope
$F_I=(\Delta_I,\ca{H}_I)$ in the space $H_I\subset V^*$. Note that
the projected multi-fan $\Delta_I$ is defined in the space $V_I$
(see subsection \ref{subsecCompleteMFans}), so the multi-polytope
based on $\Delta_I$ should formally lie in $V_I^*$. Nevertheless,
we may identify $H_I$ with $V_I^*$. The supporting hyperplanes of
$F_I$ are defined as follows: $\ca{H}_I(j)=H_I\cap H_j$ for any
vertex $j$ of $\lk_KI$. The multi-polytope $F_I$ is called
\emph{the face} of $P$ dual to $I$.
\end{defin}

%
%

\subsection{Duistermaat--Heckman function of a
multi-polytope}\label{subsecDHfunc}

Suppose $I\in K\indn$. Then the set $\{\lambda(i)\mid i\in I\}$ is
a basis of $V$. Denote its dual basis of $V^*$ by $\{u_i^I\mid
i\in I\}$, i.e. $\langle u_i^I,\lambda(j)\rangle=\delta_{ij}$
where $\delta_{ij}$ denotes the Kronecker delta. Take a generic
vector $v\in V$. Then $\langle u_i^I,v\rangle\neq 0$ for all $I\in
K\indn$ and $i\in I$. Set
\[
(-1)^I:=(-1)^{\sharp\{i\in I\mid\langle
u_i^I,v\rangle>0\}}\quad\mbox{ and }\quad (u_i^I)^+:=\begin{cases}
u_i^I\mbox{ if }\langle u_i^I,v\rangle>0,\\
-u_i^I\mbox{ if }\langle u_i^I,v\rangle<0.
\end{cases}
\]
We denote by $C^*(I)^+$ the cone in $V^*$ spanned by $(u_i^I)^+$'s
($i\in I$) with apex at a vertex $H_I$ of a multi-polytope $P$,
and by $\phi_I$ the function on $V^*$ which takes value $1$ inside
$C^*(I)^+$ and $0$ outside (this is just a characteristic function
of a subset but we want to avoid this term since it is already
reserved for the function $\lambda$).

\begin{defin}
A function $\Dh_P$ on $V^*\setminus \bigcup_{i=1}^mH_i$ defined by
\[
\sum_{I\in K\indn} (-1)^Iw(I)\phi_I
\]
is called \emph{a Duistermaat--Heckman function} associated with
$P$.
\end{defin}

The summands in the definition depend on the choice of a generic
vector $v\in V$. Nevertheless, the function itself is independent
of $v$ when $\Delta$ is complete (we refer to \cite{HM} when $w$
is integral-valued and note that the same argument works for real
weights).

The function $\Dh_P$ for a simple multi-polytope $P$ based on a
complete multi-fan has the following geometrical interpretation.
Let $S$ be the realization of first barycentric subdivision of $K$
and let $G_I\subset S$ be the dual face of $I\in K$,
$I\neq\varnothing$, i.e. a realization of the set
$\{\{I<I_1<\cdots<I_k\}\in K'\}$. If $I\in K\indn$, then $G_I$ is
a point. For a given multi-polytope $P$ based on $\Delta$ there
exists a continuous map $\psi\colon S\to V^*$ such that
$\psi(G_I)\subset H_I$ for any $I\in K$, $I\neq\varnothing$ (in
particular, when $I\in K\indn$, this map sends the point $G_I\in
S$ to the vertex $H_I$ of a multi-polytope $P$). This map is
unique up to homotopy preserving the stratifications.

Let us take any point $u\in V^*\setminus \bigcup_{i=1}^mH_i$. Then
$u$ is not contained in the image of $\psi$ by the construction of
$\psi$. Thus we may consider the induced map in homology:
\[
\psi_*\colon \Hr_{n-1}(S;\Ro)\to \Hr_{n-1}(V^*\setminus
\{u\};\Ro).
\]
The underlying simplicial cycle $[\Delta]$ may be considered as an
element of the group $\Hr_{n-1}(S;\Ro)$. Since $V^*$ is oriented,
we have the fundamental class $[V^*\setminus \{u\}]\in
\Hr_{n-1}(V^*\setminus \{u\};\Ro)$. Thus
\[
\psi_*([\Delta])=\WN_P(u)\cdot[V^*\setminus \{u\}],
\]
for some number $\WN_P(u)\in \Ro$. This number has a natural
meaning of winding number of cycle $[\Delta]$ around $u$. It
happens that this number is exactly the value of $\Dh_P$ at the
point $u\in V^*$ (see details in \cite[Sec.6]{HM}).

It is easily seen from the above consideration that $\Dh_P$ has a
compact support when $\Delta$ is complete. Thus in the case of
complete multi-fan we may define \emph{the volume of a
multi-polytope} $P$ as
\begin{equation}\label{eqDefVolMPoly}
\vol(P) = \int_{V^*}\Dh_P(u)d\mu
\end{equation}
with respect to some euclidean measure on $V^*$ (in a presence of
a lattice $N\subset V$ the measure is normalized so that the
fundamental domain of $N^*\subset V^*$ has volume~$1$).

Finally, we may consider the volume as a function on the space
$\Poly(\Delta)\cong \Ro^m$ of analogous multi-polytopes. We have a
function $V_\Delta\colon \Ro^m\to \Ro$ whose value at
$(c_1,\ldots,c_m)$ equals $\vol(P)$ for the multi-polytope $P$
with the support parameters $c_1,\ldots,c_m$. The goal of the next
section is to study this function using equivariant localization
ideas and prove Theorem \ref{thmVolPoly}.

\begin{rem}
Needless to say that in case of actual simple convex polytopes the
notions introduced above coincide with the classical ones. If $P$
is a simple convex polytope and $\Delta$ is its normal fan, then
$\Dh_P$ takes the value $1$ inside $P$ and $0$ outside. The volume
of $P$ is just the usual volume. Note that even if $\Delta$ is an
actual fan, not all multi-polytopes based on $\Delta$ are actual
convex polytopes. Nevertheless, the notion of volume and
Duistermaat--Heckman function have transparent geometrical
meanings for all of them.
\end{rem}

\begin{ex}
Consider the two-dimensional multi-fan $\Delta$ with $m=5$ and
$V=\Ro^2$ depicted on Fig.\ref{figMFanStar}, left. Its
characteristic function is the following: $\lambda(1)=(1,0)$,
$\lambda(2)=(-2,1)$, $\lambda(3)=(1,-2)$, $\lambda(4)=(0,1)$,
$\lambda(5)=(-1,-1)$. The weight function takes the value $1$ on
the subsets $\{1,2\}$, $\{2,3\}$, $\{3,4\}$, $\{4,5\}$, $\{1,5\}$
and the value $0$ on all other subsets. Geometrically this
indicates the fact that in the multi-fan we have the cones
generated by $\{\lambda(1),\lambda(2)\}$,
$\{\lambda(2),\lambda(3)\}$, etc. with multiplicity one, and do
not have the cones generated by $\{\lambda(1),\lambda(3)\}$,
$\{\lambda(1),\lambda(4)\}$, and so on. It can be seen that every
generic point of $V=\Ro^2$ is covered by exactly two cones,
therefore $\Delta$ is pre-complete of degree $2$. Moreover, a
simple check shows that all its projected multi-fans are complete.
Hence $\Delta$ is complete. The underlying chain of $\Delta$ has
the form $(1,2)+(2,3)+(3,4)+(4,5)+(5,1)\in
C_1(\triangle_{[5]}^{(1)};\Ro)$ which is obviously a simplicial
cycle. The underlying complex $K$ of $\Delta$ is a circle made of
$5$ segments, and $[\Delta]$ is its fundamental class.

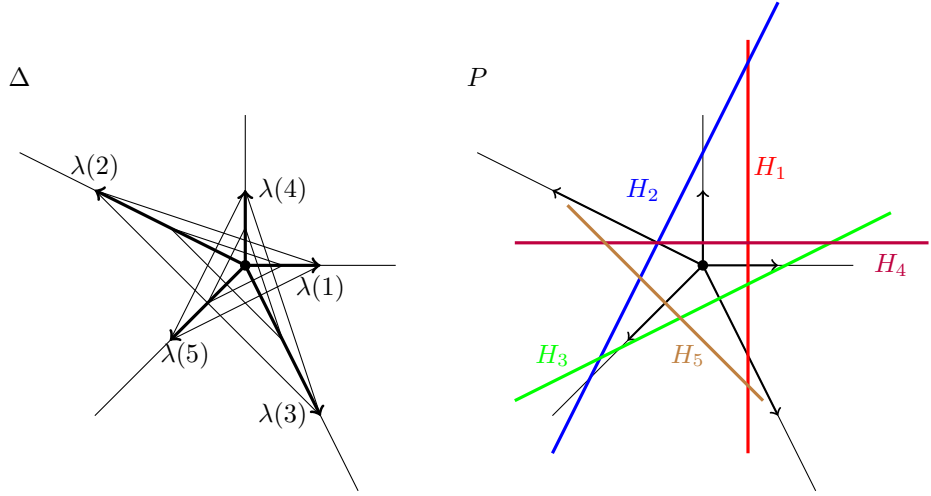
\begin{figure}[h]
\begin{center}
    \begin{tikzpicture}[scale=1]
        \draw (0,0)--(2,0);
        \draw (0,0)--(-3,1.5);
        \draw (0,0)--(1.5,-3);
        \draw (0,0)--(0,2);
        \draw (0,0)--(-2,-2);
        \draw[->, very thick] (0,0)--(1,0);
        \draw[->, very thick] (0,0)--(-2,1);
        \draw[->, very thick] (0,0)--(1,-2);
        \draw[->, very thick] (0,0)--(0,1);
        \draw[->, very thick] (0,0)--(-1,-1);
        \fill (0,0) circle (2pt);

        \draw (1,-0.3) node{$\lambda(1)$};
        \draw (-2,1.3) node{$\lambda(2)$};
        \draw (0.5,-2) node{$\lambda(3)$};
        \draw (0.5,1) node{$\lambda(4)$};
        \draw (-0.8,-1.2) node{$\lambda(5)$};

        \draw (1,0)--(-2,1);
        \draw (0.5,0)--(-1,0.5);
        \draw (-2,1)--(1,-2);
        \draw (-1,0.5)--(0.5,-1);
        \draw (1,-2)--(0,1);
        \draw (0.5,-1)--(0,0.5);
        \draw (0,1)--(-1,-1);
        \draw (0,0.5)--(-0.5,-0.5);
        \draw (-1,-1)--(1,0);
        \draw (-0.5,-0.5)--(0.5,0);

        \draw (-3,2.5) node{$\Delta$};
    \end{tikzpicture}
    \qquad
    \begin{tikzpicture}[scale=1]

        \draw (0,0)--(2,0);
        \draw (0,0)--(-3,1.5);
        \draw (0,0)--(1.5,-3);
        \draw (0,0)--(0,2);
        \draw (0,0)--(-2,-2);
        \draw[->, thick] (0,0)--(1,0);
        \draw[->, thick] (0,0)--(-2,1);
        \draw[->, thick] (0,0)--(1,-2);
        \draw[->, thick] (0,0)--(0,1);
        \draw[->, thick] (0,0)--(-1,-1);
        \fill (0,0) circle (2pt);

        \draw[very thick, red] (0.6,-2.5)--(0.6,3);
        \draw[very thick, blue] (1,3.5)--(-2,-2.5);
        \draw[very thick, green] (2.5,0.7)--(-2.5,-1.8);
        \draw[very thick, purple] (-2.5,0.3)--(3,0.3);
        \draw[very thick, brown] (-1.8,0.8)--(0.8,-1.8);

        \draw[red] (0.9,1.3) node{$H_1$};
        \draw[blue] (-0.8,1) node{$H_2$};
        \draw[green] (-2,-1.2) node{$H_3$};
        \draw[purple] (2.5,0) node{$H_4$};
        \draw[brown] (-0.2,-1.2) node{$H_5$};

        \draw (-3,2.5) node{$P$};
    \end{tikzpicture}
\end{center}
\caption{Example of a multi-fan $\Delta$ and a multi-polytope $P$
based on it.} \label{figMFanStar}
\end{figure}

An example of a multi-polytope $P$ based on $\Delta$ is shown at
Fig.\ref{figMFanStar}, right. Each hyperplane $H_i$ is orthogonal
to the linear span of the corresponding ray $\lambda(i)$ of
$\Delta$, $i\in [5]$. The Duistermaat--Heckman function of $P$ is
shown on Fig.\ref{figMPolyDH}. The function is constant on the
chambers: it takes value $2$ in the middle pentagon since the
multi-polytope ``winds'' around the points of this region twice,
and takes value $1$ on triangles adjacent to the central pentagon.
The value of $\Dh_P$ in all other chambers is $0$. The volume of a
multi-polytope is therefore \emph{not} just the volume of the
five-point star: the points in the central region contribute to
the volume twice.
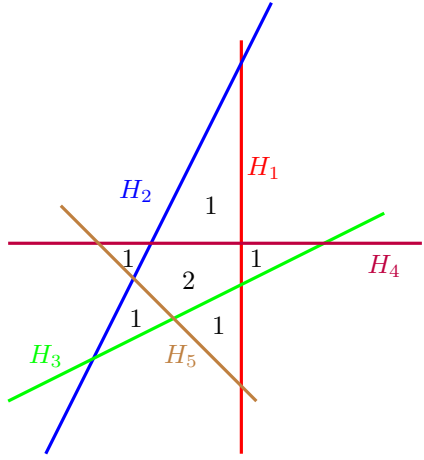
\begin{figure}[h]
\begin{center}
    \begin{tikzpicture}[scale=1]


        \draw[very thick, red] (0.6,-2.5)--(0.6,3);
        \draw[very thick, blue] (1,3.5)--(-2,-2.5);
        \draw[very thick, green] (2.5,0.7)--(-2.5,-1.8);
        \draw[very thick, purple] (-2.5,0.3)--(3,0.3);
        \draw[very thick, brown] (-1.8,0.8)--(0.8,-1.8);

        \draw[red] (0.9,1.3) node{$H_1$};
        \draw[blue] (-0.8,1) node{$H_2$};
        \draw[green] (-2,-1.2) node{$H_3$};
        \draw[purple] (2.5,0) node{$H_4$};
        \draw[brown] (-0.2,-1.2) node{$H_5$};

        \draw (-0.1,-0.2) node{$2$};
        \draw (-0.9,0.1) node{$1$};
        \draw (0.2,0.8) node{$1$};
        \draw (0.8,0.1) node{$1$};
        \draw (0.3,-0.8) node{$1$};
        \draw (-0.8,-0.7) node{$1$};

    \end{tikzpicture}
\end{center}
\caption{Duistermaat--Heckman function of the multi-polytope $P$.}
\label{figMPolyDH}
\end{figure}

\end{ex}

%
%
%
%
%
%

\section{Volume polynomial from the index
map}\label{secVolPolyIndexMap}

%
%

\subsection{Index map}\label{subsecIndexMap}

Let $\Delta=(w_{ch},\lambda)$ be a simplicial multi-fan in $V\cong
\Ro^n$ with $m$ rays. The characteristic function
$\lambda\colon[m]\to V$ may be considered as a linear map
$\lambda\colon \Ro^m\to V$ which sends the basis vector $e_i\in
\Ro^m$, $i\in [m]$ to $\lambda(i)$. Let $\{x_i\}_{i\in[m]}$ be the
basis of $(\Ro^m)^*$ dual to $\{e_i\}_{i\in[m]}$, so that
$(\Ro^m)^*=\langle x_1,\ldots,x_m\rangle$. Let us also consider
the adjoint map $\lambda^\top\colon V^*\to (\Ro^m)^*$. By
definition it sends the vector $u\in V^*$ to
\[
\sum_{i=1}^m\langle u,\lambda(i)\rangle x_i.
\]
For any maximal simplex $I=\{i_1,\ldots,i_n\}\in K\indn$ the
vectors $\{\lambda(i)\}_{i\in I}$ form a basis of $V$ according to
$*$-condition, defined in subsection
\ref{subsecMFansWeightCharFunc}. Let $\{u_i^I\}_{i\in I}$ be the
dual basis of $V^*$. Let $\iota_I\colon (\Ro^m)^*\to V^*$ be the
linear map defined by
\begin{equation}\label{eqIotaVanish}
\iota_I(x_i)=\begin{cases}u^I_i\mbox{  if } i\in I,\\
0,\mbox{ if }i\notin K.\end{cases}
\end{equation}
%
%
%
Consider $\Ro[x_1,\ldots,x_m]$, the algebra of polynomials on
$(\Ro^m)^*$. Also let $\Ro[V^*]$ denote the algebra of polynomials
on $V^*$. Both polynomial algebras are graded, where we set the
degrees of the generating spaces $(\Ro^m)^*$ and $V^*$ to $2$. The
linear map $\iota_I$ induces the graded algebra homomorphism
\[
\iota_I\colon \Ro[x_1,\ldots,x_m]\to \Ro[V^*],
\]
denoted by the same letter. In the following, if $A$ is a graded
algebra, we denote by $A_j$ its homogeneous part of degree $j$.

Let $S^{-1}\Ro[V^*]$ denote the ring of rational functions over
$\Ro[V^*]$ graded in a natural way. Given a weight function
$w\colon K\indn\to \Ro$ we can define the linear map
$\pi_!^\Delta\colon\Ro[x_1,\ldots,x_m]\to S^{-1}\Ro[V^*]$ as the
following weighted sum:
\begin{equation}\label{eqDefIndexMap}
\pi_!^\Delta(x)=\sum_{I\in K\indn}\dfrac{w(I) \iota_I(x)}{|\det
\lambda_I|\prod_{i\in I}\iota_I(x_i)}
\end{equation}
for $x\in \Ro[x_1,\ldots,x_m]$. We assume that an inner product is
fixed on $V$, so that $|\det\lambda_I| = |\det(\lambda(i)_{i\in
I})|$ is well-defined even if there is no lattice in $V$. The
inner product on $V$ induces a euclidean measure on $V^*$ and
$|\det\lambda_I|$ is the volume of the parallelepiped spanned by
$\{\lambda(i)\}_{i\in I}$. The translation invariant measure on
$V^*$ is assumed the same as in \eqref{eqDefVolMPoly}. The map
$\pi_!^\Delta$ is well-defined since $\lambda_I$ are isomorphisms.
It can be seen that $\pi_!^\Delta$ is homogeneous of degree $-2n$.
It is called \emph{the index map of multi-fan}
$\Delta=(K,w,\lambda)$.

\begin{thm}\label{thmIndexMap}
The following properties of $\Delta$ are equivalent:
\begin{enumerate}
\item The image of $\pi_!^\Delta$ lies in $\Ro[V^*]\subset
S^{-1}\Ro[V^*]$;

\item The underlying chain $w_{ch}=\sum_{I\in K\indn}w(I)I$ is
closed;

\item The multi-fan $\Delta=(w_{ch},\lambda)$ is complete.
\end{enumerate}
\end{thm}

\begin{proof}
Equivalence of (2) and (3) was already shown in Proposition
\ref{propCompleteCriterion}. The implication (2) $\Rightarrow$
(1), in case when $\lambda$ takes values in the lattice and $w$ is
integer-valued, is proved in \cite[Lm.8.4]{HM}. It should be noted
that in this case $|\det\lambda_I|$ appearing in the denominator
is nothing but the order of the finite group $G_I=N/N_I$, where
$N\subset V$ is the lattice and $N_I$ is a sublattice generated by
$\{\lambda(i)\}_{i\in I}$. The situation when $\lambda$ and $w$
are rational is reduced to the integral case by multiplying all
values of $\lambda$ and $w$ by a common denominator (both
conditions (1) and (2) are invariant under rescaling). The real
case follows by continuity. Indeed, the subset of simplicial
cycles with rational coefficients, $Z_{n-1}(K;\Qo)$, is dense in
$Z_{n-1}(K;\Ro)$; the right hand side of \eqref{eqDefIndexMap} is
continuous with respect to $\lambda$ and $w$; and the subset
$\Ro[V^*]$ is closed in $S^{-1}\Ro[V^*]$. Therefore, arbitrary
complete multi-fan $(w,\lambda)$ can be approximated by a sequence
of rational complete multi-fans
$\Delta_\alpha=(w_\alpha,\lambda_\alpha)$ which implies that the
values of $\pi_!^\Delta$ are approximated by the values of
$\pi_!^{\Delta_\alpha}$. Since the values of
$\pi_!^{\Delta_\alpha}$ are polynomials, so are the values of
$\pi_!^\Delta$.

Let us prove that (1) implies (2). Take any simplex $J\in K$ such
that $|J|=n-1$ and consider the monomial $x_J=\prod_{i\in J}x_j$
of degree $2(n-1)$ lying in $\Ro[x_1,\ldots,x_m]$. The map
$\pi_!^\Delta$ lowers the degree by $2n$ thus we have
$\deg\pi_!^\Delta(x_J)=-2$. Condition (1) implies that
$\pi_!^\Delta(x_J)$ is a polynomial, thus $\pi_!^\Delta(x_J)=0$.
By definition, we have
\[
\pi_!^\Delta(x_J)=\sum_{I\in
K\indn}\dfrac{w(I)\iota_I(x_J)}{|\det\lambda_I|\prod_{i\in
I}\iota_I(x_i)}
\]
Note that $\iota_I$ is a ring homomorphism and $\iota_I(x_i)=0$ if
$i\notin I$ by \eqref{eqIotaVanish}. Therefore,
\[
\pi_!^\Delta(x_J)=\sum_{I\in K\indn, J\subset
I}\dfrac{w(I)\prod_{i\in
J}\iota_I(x_i)}{|\det\lambda_I|\prod_{i\in I}\iota_I(x_i)} =
\sum_{j\in[m]\setminus J, I:=J\cup \{j\}\in
K\indn}\dfrac{w(I)}{|\det\lambda_I|\iota_I(x_j)}
\]
Recall that $\iota_I(x_i)=u_i^I$, where $\{u_i^I\}_{i\in I}$ is
the basis of $V^*$ dual to the basis $\{\lambda(i)\}_{i\in I}$ of
$V$. Consider the linear functional $\varrho\in V^*$ taking the
value $\varrho(v)=\det((\lambda(i))_{i\in J},v)$ for any $v\in V$.
It can be seen that $|\det\lambda_I|\iota_I(x_j)\in V^*$, where
$I=J\cup\{j\}$, coincides with $\varrho$ up to sign. More
precisely $|\det\lambda_I|\iota_I(x_j)=\inc{I}{J} \varrho$, where
$\inc{I}{J}$ is the incidence sign of two simplices of $K$ (it
appears because we need to permute the vectors
$((\lambda(i))_{i\in J},\lambda(j))$ in order to get the positive
determinant). Therefore,
\[
0=\pi_!^\Delta(x_J)=\dfrac{1}{\varrho}\sum_{I\in K\indn, J\subset
I} \inc{I}{J}w(I)
\]
It remains to notice that the sum in this expression is exactly
the coefficient of $J$ in the simplicial chain $dw_{ch}\in
C_{n-2}(K;\Ro)$. This calculation applies to any $J\in K$,
$|J|=n-2$, therefore $dw_{ch}=0$.
\end{proof}

The map $\lambda^{\top}\colon V^*\to (\Ro^m)^*$, the adjoint of
$\lambda$, induces the ring homomorphism $\Ro[V^*]\to
\Ro[x_1,\ldots,x_m]$. Hence $\Ro[x_1,\ldots,x_m]$ obtains the
structure of $\Ro[V^*]$-module. It can be checked that
$\lambda^\top$ is the right inverse of each $\iota_I\colon
(\Ro^m)^*\to V^*$, therefore all ring homomorphisms $\iota_I\colon
\Ro[x_1,\ldots,x_m]\to\Ro[V^*]$ are the $\Ro[V^*]$-module
homomorphisms. Thus $\pi_!^\Delta$ is also a homomorphism of
$\Ro[V^*]$-modules (even in the case $w_{ch}$ is not closed).

\begin{rem}
Note that conditions (1) and (2) in Theorem \ref{thmIndexMap} make
sense over an arbitrary field $\ko$. We may start with a
$\ko$-valued chain $w_{ch}\in C_{n-1}(K;\ko)$ and a characteristic
function valued in $\ko^n$. These data allow to define the maps
$\iota_I$ and $\pi_!^\Delta$ absolutely similar to the real case.

\begin{problem}
Does equivalence of (1) and (2) in Theorem \ref{thmIndexMap} hold
for arbitrary fields?
\end{problem}

For general fields we cannot reduce the task to the integral case
but it is likely that there exists a straightforward algebraical
proof.
\end{rem}

%
%

\subsection{Stanley--Reisner rings}\label{subsecSRrings}

Let us recall the definition of the Stanley--Reisner ring.

\begin{defin}
Let $K$ be a simplicial complex on the vertex set $[m]$ and $\ko$
be a ground ring (either $\Zo$ or a field). The Stanley--Reisner
ring is the quotient of a polynomial ring by the Stanley--Reisner
ideal:
\[
\ko[K]:=\ko[x_1,\ldots,x_m]/I_{SR},\mbox{ where }
I_{SR}=(x_{i_1}\cdot\ldots\cdot x_{i_k}\mid
\{i_1,\ldots,i_k\}\notin K),
\]
endowed with the grading $\deg x_i=2$ and the natural structure of
graded $\ko[x_1,\ldots,x_m]$-module.
\end{defin}

For now let us concentrate on the case $\ko=\Ro$. Given a
characteristic function $\lambda$ on $K$ we may define a certain
ideal in $\Ro[K]$ generated by linear forms. As before, let
$\lambda^\top\colon V^*\to (\Ro^m)^*=\langle
x_1,\ldots,x_m\rangle$ denote the adjoint map of $\lambda\colon
\Ro^m\to V$. Let $\Theta$ denote the ideal of
$\Ro[x_1,\ldots,x_m]$ generated by the image of $\lambda^\top$. By
abuse of notation we denote the corresponding ideal in $\Ro[K]$
with the same letter $\Theta$.

Let us state things in the coordinate form. Fix a basis
$f_1,\ldots,f_n$ of $V$. Then every characteristic value
$\lambda(i)$, $i\in[m]$ is written as a row-vector
$(\lambda_{i,1},\ldots,\lambda_{i,n})$, where $\lambda_{i,j}\in
\Ro$. The $*$-condition for $\lambda$ (see subsection
\ref{subsecMFansWeightCharFunc}) states that the square matrix
formed by row-vectors $(\lambda_{i,1},\ldots,\lambda_{i,n})_{i\in
I}$ is non-degenerate for any $I\in K\indn$.

If we consider the dual basis $\bar{f}_1,\ldots,\bar{f}_n$ in the
dual space $V^*$, then its image under $\lambda^\top\colon V^*\to
(\Ro^m)^*=\langle x_1,\ldots,x_m\rangle$ has the form
\[
\theta_j:=\lambda^\top(\bar{f}_j)=\lambda_{1,j}x_1+\lambda_{2,j}x_2+\cdots+\lambda_{m,j}x_m
\]
for $j=1,\ldots,n$. Thus $\Theta$ (as an ideal either in
$\Ro[x_1,\ldots,x_m]$ or $\Ro[K]$) is generated by the elements
$\theta_1,\ldots,\theta_n$. In particular, if $\lambda$ is
integer-valued, then $\Theta=(\theta_1,\ldots,\theta_n)$ may be
considered as a well-defined ideal in $\Zo[K]$ or
$\Zo[x_1,\ldots,x_m]$.

It is known that the Krull dimension of $\Ro[K]$ equals $\dim K+1
= n$ (see e.g. \cite{St}), and $\theta_1,\ldots,\theta_n$ is a
linear system of parameters in $\Ro[K]$ for any characteristic
function $\lambda$ and every choice of a basis in $V$ (e.g.
\cite[Lm.3.3.2]{BPnew}). Thus $\Ro[K]/\Theta$ has Krull dimension
$0$, which in our case is equivalent to saying that
$\Ro[K]/\Theta$ is a finite-dimensional vector space. Moreover, it
is known (see e.g. \cite[Lm.8.1]{HM} or \cite[Lm.3.5]{AyV3}) that
the classes of monomials $x_I=x_{i_1}\cdot\ldots\cdot x_{i_k}$
taken for each simplex $I=\{i_1,\ldots,i_k\}\in K$ linearly span
$\Ro[K]/\Theta$ (however there exist relations on these classes!).

We introduce the following notation to make the exposition
consistent with that of \cite{HM}:
\begin{equation}
H_T^*(\Delta;\ko):=\ko[K],\qquad H^*(\Delta;\ko):=\ko[K]/\Theta,
\end{equation}
and, for short, $H^*_T(\Delta):=H^*_T(\Delta;\Ro)$ and
$H^*(\Delta):=H^*(\Delta;\Ro)$.

%
%

\subsection{Evaluation on fundamental
class}\label{subsecIntegration}

Let $x=x_{i_1}^{j_1}\cdot\ldots\cdot x_{i_k}^{j_k}$ be a monomial
whose index set $\{i_1,\ldots,i_k\}$ is not a simplex of $K$. Then
$\iota_I(x)=0$ for any $I\in K\indn$, according to
\eqref{eqIotaVanish}. Therefore $\pi_!^\Delta(x)=0$. Hence
$\pi_!^\Delta$ vanishes on the Stanley--Reisner ideal $I_{SR}$ and
descends to the map
\[
\pi_!^\Delta\colon \Ro[K]\to \Ro[V^*].
\]
Since $\pi_!^\Delta$ is a map of $\Ro[V^*]$-modules, we may apply
$\otimes_{\Ro[V^*]}\Ro$ to $\pi_!^\Delta$. This gives a linear map
\[
\int_\Delta\colon H^{2n}(\Delta)\to \Ro\cong\Ro[V^*]/\Ro[V^*]^+
\]

\begin{defin}
Let $\Delta=(K,w,\lambda)$ be a complete simplicial multi-fan. The
map $\int_\Delta\colon H^{2n}(\Delta)\to \Ro$ is called
``\emph{the evaluation on the fundamental class of $\Delta$}''.
\end{defin}

We denote the composite map $\Ro[x_1,\ldots,x_m]\twoheadrightarrow
H^{2n}(\Delta)\to\Ro$ by $\int_{\Delta,\Ro[m]}$.

%
%

\subsection{Chern class of a
multi-polytope}\label{subsecChernClassMPoly}

Let $P$ be a multi-polytope based on a complete simplicial
multi-fan $\Delta=(K,w,\lambda)$ of dimension $n$, and let
$c_1,\ldots,c_m\in \Ro$ be the support parameters of $P$. The
element
\[
c_1(P):=c_1x_1+\cdots+c_mx_m\in H^2(\Delta)
\]
is called \emph{the first Chern class of $P$}.

\begin{prop}\label{propVolFormula}
\begin{equation}\label{eqVolFormula}
\vol(P)=\frac{1}{n!}\int_{\Delta}c_1(P)^n.
\end{equation}
\end{prop}

\begin{proof}
If $\lambda$ and $w$ are integral, the statement is proved in
\cite[Lm.8.6]{HM}. The rational case follows from the integral
case by the following arguments. (1) In the rational case we may
choose a refined lattice such that $\lambda$ becomes integral with
respect to this lattice (this would change the euclidean measure
on $V^*$, but this change affects both sides of
\eqref{eqVolFormula} in the same way). (2) A rational weight $w$
may be turned into an integral weight by rescaling (both sides of
\eqref{eqVolFormula} depend linearly on $w$, thus rescaling of $w$
preserves \eqref{eqVolFormula}). Real case follows by continuity,
since both sides of \eqref{eqVolFormula} depend continuously on
$\lambda$ and $w$.
\end{proof}

It is easily seen that, for a given $\Delta$, the expression on
the right hand side of \eqref{eqVolFormula} is a homogeneous
polynomial of degree $n$ in the variables $c_1,\ldots,c_m$:
\[
V_\Delta(c_1,\ldots,c_m)=\frac{1}{n!}\int_{\Delta}(c_1x_1+\cdots+c_mx_m)^n.
\]
Thus Proposition \ref{propVolFormula} implies Theorem
\ref{thmVolPoly}.

%
%
%
%
%
%

\section{Basic properties of volume
polynomials}\label{secBasicPropVolPoly}

\subsection{Partial derivatives of volume
polynomial}\label{subsecDerivsVolPoly}

We continue to assume that there is a fixed inner product in $V$
which makes the integral lattice in $V$ unnecessary. The inner
product allows to identify $V$ and $V^*$ and to introduce a
measure on each affine subspace of $V$ or $V^*$. Consider the
space $\Lambda^kV$ of exterior forms on $V$. Given an inner
product in $V$ we obtain an inner product on $\Lambda^kV$.

Suppose that every simplex $I\in K$ is oriented somehow. For a
characteristic function $\lambda\colon [m]\to V$ on $K$ and
$I=\{i_1,\ldots,i_k\}\in K$ let $\lambda(I)$ denote the skew form
$\lambda(i_1)\wedge\cdots\wedge\lambda(i_k)\in \Lambda^kV$, where
$(i_1,\ldots,i_k)$ is the positive order of vertices of $I$.
Denote the norm of $\lambda(I)$ by $\covol(I)$:
\[
\covol(I) := \|\lambda(I)\| =
\|\lambda(i_1)\wedge\cdots\wedge\lambda(i_k)\|.
\]

Recall from Section \ref{secDefMultipoly} the notion of a face of
a multi-polytope. If $P$ is a multi-polytope of dimension $n$ and
$I\in K$ then $F_I$ is a multi-polytope of dimension $n-|I|$
sitting in the affine subspace $H_I\subset V^*$. There is a
measure on $H_I$ determined by the inner product, hence we may
define the volume of $F_I$. The following lemma shows that we can
compute the volumes of faces from the volume polynomial.

\begin{lemma}[{cf.\cite[Thm.2.4.3]{Tim}}]\label{lemDerivsOfVolPol}
Let $J\subset [m]$. Consider the homogeneous polynomial
$\dd_JV_\Delta$ of degree $n-|J|$. Then
\begin{enumerate}
\item Let $\theta_u$ denote the linear differential operator
$\sum_{i=1}^m\langle u, \lambda(i)\rangle \dd_i$ for $u\in V^*$.
Then $\theta_uV_\Delta=0$.

\item If $J\notin K$, then $\dd_JV_\Delta=0$;

\item If $J\in K$, then the value of the polynomial $\dd_JV_\Delta$ at a point
$(\tilde{c}_1,\ldots,\tilde{c}_m)\in \Ro^m$ is equal to
\begin{equation}
\dfrac{\vol F_J}{\covol(J)}
\end{equation}
when $|J|<n$ and
\begin{equation}
\dfrac{w(J)}{\covol(J)}=\dfrac{w(J)}{|\det \lambda_J|}
\end{equation}
when $|J|=n$. Here $\tilde{c}_i$ are the support parameters of a
multi-polytope $P$ and $F_J$ are its faces.
\end{enumerate}
\end{lemma}

\begin{proof}
(1) We have
\[\begin{split}
\theta_uV_\Delta&=\frac{1}{n!}\int\Big(\sum_{i=1}^m\langle
u,\lambda(i)\rangle \frac{\dd}{\dd
c_i}\Big)(c_1x_1+\cdots+c_mx_m)^n\\&=\frac{1}{(n-1)!}
\int\Big(\sum_{i=1}^m\langle u, \lambda(i)\rangle
x_i\Big)\cdot(c_1x_1+\cdots+c_mx_m)^{n-1}=0,
\end{split}
\]
since $\sum_{i=1}^m\langle u, \lambda(i)\rangle x_i=0$ in
$H^*(\Delta)$.

(2) The proof of second statement is completely similar to (1). We
have
\[\begin{split}
\dd_JV_\Delta&=\frac{1}{n!}\int\Big(\prod_{i\in J}\frac{\dd}{\dd
c_i}\Big)(c_1x_1+\cdots+c_mx_m)^n\\&=\frac{1}{(n-|J|)!}
\int(\prod_{i\in J}x_i)\cdot(c_1x_1+\cdots+c_mx_m)^{n-|J|}=0,
\end{split}
\]
since $x_J=\prod_{i\in J}x_i=0$ in $H^*(\Delta)$ for $J\notin K$.

(3) The second claim requires some technical work. At first, let
$|J|=n$, i.e. $J\in K\indn$. We have
\[
\dd_JV_\Delta=\dd_J\frac{1}{n!}\int_\Delta(c_1x_1+\cdots+c_mx_m)^n=\int_\Delta
x_J=\pi_!^\Delta(x_J),
\]
where $x_J=\prod_{i\in J}x_i$. By the definition of the index map
\eqref{eqDefIndexMap} we have
\[
\pi_!^\Delta(x_J)=\sum_{I\in
K\indn}\dfrac{w(I)\iota_I(x_J)}{|\det\lambda_I|\prod_{i\in
I}\iota_I(x_i)}.
\]
If $I\neq J$, the corresponding summand vanishes, since
$\iota_I(x_j)=0$ for $j\notin I$ by~\eqref{eqIotaVanish}. The
summand corresponding to $I=J$ contributes
$\frac{w(J)}{|\det\lambda_{J}|}$ which proves the statement.

Let us prove the case $|J|<n$. Recall that the projected multi-fan
$\Delta_J=(\lk_KJ,w_J,\lambda_J)$ is the multi-fan in the vector
space $V_J=V/\langle \lambda(j)\mid j\in J\rangle$. There exists a
``restriction" map
\[
\varphi_J\colon H^*(\Delta)\to H^*(\Delta_J),
\]
defined as follows:
\[
\varphi_J(x_j)=\begin{cases} x_j,\mbox{ if }j\in \lk_KJ;\\
-\sum_{i\in \lk_KJ}p^{J}_{i,j}x_i, \mbox{ if }j\in J;\\
0,\mbox{ otherwise}
\end{cases}
\]
Here the constants $p^{J}_{i,j}$ for $j\in J$ and $i\in \lk_KJ$
are defined by
\begin{equation}\label{eqPconsts}
\proj_{J}\lambda(i)=\sum_{j\in J}p^{J}_{i,j}\lambda(j),
\end{equation}
where $\proj_{J}\lambda(i)$ is the orthogonal projection of the
vector $\lambda(i)$ to the linear subspace spanned by $\lambda(j)$
$(j\in J)$.

The homomorphism $\varphi_J$ is now defined on the level of
polynomial algebras.

\begin{claim}
$\varphi_J$ is a well-defined ring homomorphism from
$H^*(\Delta)=\Ro[K]/\Theta$ to
$H^*(\Delta_J)=\Ro[\lk_KJ]/\Theta_J$.
\end{claim}

\begin{proof}
The proof is a routine check. First let us prove that
Stanley--Reisner relations in $\Delta$ are mapped to the
Stanley--Reisner ideal of $\Delta_J$. Let $I$ be a non-simplex of
$K$. The definition of $\varphi_J$ implies that $\varphi_J(x_I)=0$
unless $I\subset J\cup\ver(\lk_KJ)$. If $I\subset
J\cup\ver(\lk_KJ)$, we have that $I\cap \ver(\lk_KJ)$ is a
non-simplex of $\lk_KJ$ (otherwise we would have $I\in K$
contradicting the assumption). Then the element
$\varphi_J(x_I)=\varphi_J\big(\prod_{i\in I\cap
\ver(\lk_KJ)}x_i\big)\cdot \varphi_J\big(\prod_{i\in I\cap
J}x_i\big)=\prod_{i\in I\cap \ver(\lk_KJ)}x_i\cdot
\varphi_J\big(\prod_{i\in I\cap J}x_i\big)$ lies in the
Stanley--Reisner ideal of $\lk_KJ$.

Let us check that linear relations in $H^*(\Delta)$ are mapped
into linear relations of $H^*(\Delta_J)$. A general linear
relation in $H^*(\Delta)$ has the form $\sum_{i\in[m]}\langle u,
\lambda(i)\rangle x_i$ for some $u\in V^*$. The map $\varphi_J$
sends it to the element
\[
\begin{split}
\sum_{i\in \ver(\lk_KJ)}&\Big(\langle u,
\lambda(i)\rangle-\sum_{j\in J}p^J_{i,j}\langle u, \lambda(j)\rangle\Big)x_i\\
&= \sum_{i\in \ver(\lk_KJ)}\Big\langle u, \lambda(i)-\sum_{j\in
J}p^J_{i,j}\lambda(j)\Big\rangle x_i=\sum_{i\in
\ver(\lk_KJ)}\langle u, \lambda_J(i)\rangle x_i.
\end{split}
\]
(note that $\lambda(i)-\sum_{j\in J}p^J_{i,j}\lambda(j) =
\lambda(i)-\proj_J\lambda(i)=\lambda_J(i)$ is the projection of
$\lambda(i)$ to the plane orthogonal to
$\langle\lambda(j)\rangle_{j\in J}$). The last expression is zero
in $H^*(\Delta_J)$.
\end{proof}

Next we show that restriction homomorphism is compatible with the
first Chern classes of the multi-polytopes.

\begin{claim}\label{claimCherns}
$\varphi_J$ sends $c_1(P)$ to $c_1(F_J)$.
\end{claim}

\begin{proof}
Recall that $H_J$ denotes the ambient space of the face $F_J$ of
the multi-polytope $P$. The supporting hyperplanes of $F_J$ are
given by intersections $H_J\cap H_i$, where $H_i$ is the
supporting hyperplane of $P$ for $i\in \lk_KJ$.

Let us denote by $U_J$ the subspace spanned by $\lambda(j)$'s
$(j\in J)$ so that $V_J=V/U_J$. By the definition (see subsection
\ref{subsecCompleteMFans}), $\lambda_J(i)$ is the projection image
of $\lambda(i)$ on $V_J$ if $i$ is the vertex of $\lk_KJ$. As in
the proof of previous claim we identify the quotient space
$V_J=V/U_J$ with the orthogonal complement $U_J^\bot$ of $U_J$.
The projected vector $\lambda_J(i)$ can be considered as the
element in $V$ and we have
\begin{equation}\label{eqOrthog}
\lambda(i)=\lambda_J(i)+\proj_J\lambda(i)
\end{equation}
with respect to the orthogonal decomposition $V=U_J^\bot\oplus
U_J$.

The affine hyperplane $H_i$ is given by $\{u\in V^*\mid \langle u,
\lambda(i)\rangle=c_i\}$. The affine plane $H_J$ is given by
$\{u\in V^*\mid \langle u, \lambda(j)\rangle=c_j,$ for all $j\in
J\}$. By using \eqref{eqPconsts} and \eqref{eqOrthog} we may write
the intersection $H_i\cap H_J$ as

\[
\begin{split}
\{u\in H_J\mid \langle u,
\lambda_J(i)+\proj_J\lambda(i)\rangle=c_i\}&=\Big\{u\in H_J \mid
\Big\langle u, \lambda_J(i)+\sum_{j\in
J}p^J_{i,j}\lambda(j)\Big\rangle=c_i\Big\}\\&= \Big\{u\in H_J\mid
\langle u, \lambda_J(i)\rangle=c_i-\sum_{j\in
J}p^J_{i,j}c_j\Big\}.
\end{split}
\]

Therefore the $i$-th support parameter of $F_J$ is $c_i-\sum_{j\in
J}p^J_{i,j}c_j$ for $i\in \lk_KJ$. Now it remains to note that the
coefficient of $x_i$ in the projected class $\varphi_J(c_1(P))$ is
exactly $c_i-\sum_{j\in J}p^J_{i,j}c_j$. Thus
$\varphi_J(c_1(P))=c_1(F_J)$.
\end{proof}

Now we prove the following

\begin{claim} \label{claimRestricHomo}
\[
\int_\Delta y\prod_{j\in
J}x_j=\frac{1}{\covol(J)}\int_{\Delta_J}\varphi_J(y)
\quad\text{for any $y\in H^*(\Delta)$.}
\]
\end{claim}

\begin{proof}
Let us denote by $\vol S$ the volume of the parallelepiped formed
by a set of vectors $S$. Then $\covol(J)=\vol\{\lambda(i)\}_{i\in
J}$ and the index map can be written as
\begin{equation} \label{eqIndexMapPpiped}
\pi_!^\Delta(x)=\sum_{I\in
K\indn}\frac{w(I)\iota_I(x)}{\vol\{\lambda(i)\}_{i\in
I}\prod_{i\in I}\iota_I(x_i)}.
\end{equation}
Let $\tilde{I}\in\lk_KJ$ and, therefore, $\tilde{I}\sqcup J\in K$.
Then
\[
\vol\{\lambda(i)\}_{i\in \tilde{I}\sqcup
J}=\vol\{\lambda(i)\}_{i\in J}\cdot\vol\{\lambda_J(i)\}_{i\in
\tilde{I}} =\covol(J)\cdot\vol\{\lambda_J(i)\}_{i\in \tilde{I}}.
\]
This together with \eqref{eqIndexMapPpiped} implies the lemma.
\end{proof}

Applying claim~\ref{claimRestricHomo} to $y=c_1(P)^{n-|J|}$ and
using claim~\ref{claimCherns}, we obtain
\[
\dd_JV_\Delta=\frac{1}{(n-|J|)!}\int_\Delta
c_1(P)^{n-|J|}\prod_{j\in J} x_j =
\frac{1}{\covol(J)}\frac{1}{(n-|J|)!}\int_{\Delta_J}c_1(F_J)^{n-|J|}.
\]
Expression at the right evaluates to $\frac{\vol(F_J)}{\covol(J)}$
which finishes the proof of Lemma~\ref{lemDerivsOfVolPol}.
\end{proof}

\begin{cor}\label{corOmegaInjective}
Let $\Delta=(w_{ch},\lambda)$ be a complete multi-fan. Then
$V_\Delta=0$ implies $w_{ch}=0$.
\end{cor}

\begin{proof}
If $V_\Delta=0$, then $\dd_JV_\Delta=0$ for any $J\in K\indn$.
This implies $w_{ch}=0$.
\end{proof}

\begin{rem}
Of course, according to Proposition \ref{propPDdescription} the
polynomial $V_\Delta$ is non-zero if and only if the map
$\int_\Delta$ is non-zero. The fact that $\int_\Delta$ is non-zero
for every non-zero $w_{ch}$ is proved by applying this map to all
monomials $x_I$, $I\in K\indn$ (recall that these monomials span
$H^{2n}(\Delta)$). This procedure is essentially the same as
applying differential operators $\dd_I$ to $V_\Delta$.
\end{rem}

\begin{cor}\label{corDerivsDerivs}
Let $\dd_P$ denote the linear differential operator
$\sum_{i\in[m]} \tilde{c}_i\dd_i$ where $\tilde{c}_i$ are the
support parameters of a multi-polytope $P$. Then
\[
\frac{1}{n!}\dd_P^nV_\Delta=\vol(P),
\]
\[
\frac{1}{(n-|I|)!}\dd_P^{n-|I|}\dd_IV_\Delta =
\dfrac{\vol(F_I)}{\covol(I)}.
\]
\end{cor}

\begin{proof}
Both formulas follow from Lemma \ref{lemDerivsOfVolPol} and a
simple observation: if $\Psi\in \Ro[c_1,\ldots,c_m]_k$ is a
homogeneous polynomial of degree $k$, then
\[
\Big(\sum\nolimits_{i\in[m]} \tilde{c}_i\dd_i\Big)^k\Psi =
k!\Psi(\tilde{c}_1,\ldots,\tilde{c}_m).
\]
(evaluation at a point coincides with the result of
differentiation up to $k!$).
\end{proof}

%
%
%

\subsection{Recovering multi-fans from volume polynomials}

When we associate a volume polynomial to a complete simplicial
multi-fan, the numbering of the one-dimensional cones by $[m]$ is
incorporated in the data of the multi-fan.  We call a multi-fan
with the numbering \emph{a based multi-fan}. Two based multi-fans
$\Delta$ and $\Delta'$ are said to be equivalent if there is an
automorphism of $V$ which induces an isomorphism between $\Delta$
and $\Delta'$ preserving the numbering. In the presence of a
lattice $N\subset V$ there should be an automorphism of the
lattice with this property. Equivalent complete simplicial based
multi-fans have the same volume polynomial. We will see that the
converse holds for complete simplicial based multi-fans $\Delta$
whose underlying simplicial complexes are oriented strongly
connected pseudo-manifolds. Strong connectedness of $K$ means that
for any two maximal simplices $I,I'\in K\indn$ there exists a
sequence of maximal simplices $I=I_0,I_1,\ldots,I_k=I'$ such that
$|I_s\cap I_{s+1}|=n-1$ for $0\leqslant s\leqslant k-1$.

We assume that the volume polynomial $V_\Delta$ associated to
$\Delta$ is non-zero. Then the class $[\Delta]$ is non-zero. Since
$K$ is assumed to be a pseudo-manifold, $w(I)\neq 0$ for any $I\in
K\indn$. Then Lemma \ref{lemDerivsOfVolPol} shows that $V_\Delta$
recovers $K$.

Remember that
\begin{equation} \label{eqLinRelInCohomol}
\sum_{i=1}^m\langle u,\lambda(i)\rangle x_i=0 \quad\text{ in
$H^*(\Delta)$ for any $u\in V^*$}.
\end{equation}
Let $J\in K$, $|J|=n-1$.  Since $K$ is assumed to be a
pseudo-manifold, there are exactly two elements $i_1$ and $i_2$ in
$[m]$ such that $J\cup\{i_1\}$ and $J\cup\{i_2\}$ are in $K\indn$.
Multiplying $x_J=\prod_{i\in J}x_i$ to the both sides in
\eqref{eqLinRelInCohomol}, we obtain
\[
\sum_{j\in J}\langle u,\lambda(j)\rangle x_jx_J+\langle
u,\lambda(i_1)\rangle x_{i_1}x_J+\langle u,\lambda(i_2)\rangle
x_{i_2}x_J=0 \quad \text{for all $u\in V^*$}.
\]
Applying $\int_\Delta$ to the above identity, we have
\[
\Big\langle u,  \sum_{j\in J}\big(\int_\Delta
x_jx_J\big)\lambda(j)+\big(\int_\Delta
x_{i_1}x_J\big)\lambda(i_1)+ \big(\int_\Delta
x_{i_2}x_J\big)\lambda(i_2)\Big\rangle=0.
\]
Since this holds for all $u\in V^*$, one can conclude
\begin{equation}\label{eqRecover}
\sum_{j\in J}\big(\int_\Delta
x_jx_J\big)\lambda(j)+\big(\int_\Delta
x_{i_1}x_J\big)\lambda(i_1)+ \big(\int_\Delta
x_{i_2}x_J\big)\lambda(i_2)=0.
\end{equation}
Note that the numbers $\int_\Delta x_{i_1}x_J$ and $\int_\Delta
x_{i_2}x_J$ are non-zero. Identity \eqref{eqRecover} shows that
once basis vectors $\{\lambda(i)\}_{i\in I}$ for some $I\in
K\indn$ are determined, then the other vectors $\lambda(k)$'s will
be determined by the intersection numbers
$\int_{\Delta}x_{\mathcal I}$ where $\mathcal I$ consists of
elements in $[m]$ with $|\mathcal I|=n$ (an element in $\mathcal
I$ may appear more than once). On the other hand, since
\[
V_\Delta=\frac{1}{n!}\int_\Delta (c_1x_1+\dots+c_mx_m)^n,
\]
the coefficient of $c^{\mathcal I}$ agrees with
$\int_{\Delta}x_{\mathcal I}$ up to some non-zero constant
independent of $\Delta$.  These show that $V_\Delta$ determines
$\Delta$ up to equivalence.

\begin{prop} \label{propVolPolyRigid}
Two complete simplicial toric varieties are isomorphic if and only
if their volume polynomials agree up to permutations of variables.
Here it is assumed that all $\lambda(i)$'s are the primitive
generators of the rays.
\end{prop}

\begin{proof}
This follows from the above observation and the fact that two
toric varieties are isomorphic if and only if their fans are
isomorphic \cite{Berc} \footnote{We are grateful to Ivan
Arzhantsev from whom we learned this fact}.
\end{proof}

%
%
%
%
%
%

\section{A formula for the volume
polynomial}\label{secFormulaVolPoly}

We say that the set $S$ of $n+1$ vectors in $V\cong\Ro^n$ is in
general position, if any $n$ of them are linearly independent. Any
such set determines a multi-fan whose underlying simplicial
complex is a boundary of a simplex $K=\dd \triangle_{[n+1]}$. The
weights of all maximal simplices are the same up to sign due to
closedness condition $dw_{ch}=0$. Thus without loss of generality
we may assume that all weights are $\pm1$ depending on the
orientations. We call such multi-fan \emph{an elementary
multi-fan} and denote it $\Delta^{el}(S)$.

\begin{lemma}\label{lemElementaryMultifan}
Let $\Delta$ be an elementary multi-fan determined by the vectors
$\lambda(1), \ldots, \lambda(n+1)\in V$. Let
$0\neq(\alpha_1,\ldots,\alpha_{n+1})\in \Ro^{n+1}$ be a nonzero
linear relation on these vectors, i.e.
$\sum_{i=1}^{n+1}\alpha_i\lambda(i)=0$. Then
\begin{equation}\label{eqPowerLinForm}
V_{\Delta}(c_1,\ldots,c_{n+1}) =
\const\cdot(\alpha_1c_1+\cdots+\alpha_{n+1}c_{n+1})^n.
\end{equation}
for some constant $\const$.
\end{lemma}

We postpone the proof to subsection \ref{subsecStructSpheres}.

\begin{rem}
It is not difficult to compute the constant: just apply the
differential operator $\dd_J$ for $J\subset[n+1]$, $|J|=n$ to both
sides of \eqref{eqPowerLinForm} and use
Lemma~\ref{lemDerivsOfVolPol}. However, we do not need this
constant at the moment and ignore it to simplify the exposition.
\end{rem}

\begin{thm}\label{thmVolPolyFormula}
Let $\Delta=(w_{ch},\lambda)$ be a complete multi-fan. Let $v\in
V$ be a generic vector. Then
\begin{equation}\label{eqVolPolyFormula}
V_{\Delta}(c_1,\ldots,c_m)=\frac{1}{n!}\sum_{I=\{i_1,\ldots,i_n\}\in
K} \frac{w(I)}{|\det\lambda_I|\prod_{j=1}^n\alpha_{I,j}}
(\alpha_{I,1}c_{i_1}+\cdots+\alpha_{I,n}c_{i_n})^n,
\end{equation}
where $\alpha_{I,1},\ldots,\alpha_{I,n}$ are the coordinates of
$v$ in the basis $(\lambda(i_1),\ldots,\lambda(i_n))$, and $w(I)$
is the weight.
\end{thm}

\begin{proof}
We derive a more general family of formulas, and
\eqref{eqVolPolyFormula} will be a particular case. Let $[m']$ be
a set containing $[m]$ and let
\[
z_{ch}=\sum_{J\in {[m']\choose n+1}}z(J)J\in
C_{n}(\triangle_{[m']};\Ro)
\]
be a simplicial chain such that $dz_{ch} = w_{ch}$ (it exists
since $w_{ch}$, considered as an element in
$C_{n}(\triangle_{[m']};\Ro)$, is closed hence exact). Consider
any function $\eta\colon[m']\to V$ which extends
$\lambda\colon[m]\to V$ and satisfies the condition: for any
$J=\{j_1,\ldots,j_{n+1}\}$ with $z(J)\neq 0$ the vectors
$\eta(j_1),\ldots,\eta(j_{n+1})$ are in general position. Thus for
any such $J$ we can construct an elementary multi-fan
$\Delta^{el}(\eta(J))$.

In the group of multi-fans we have a relation $\Delta = \sum_{J\in
{[m']\choose n+1}}z(J)\Delta^{el}(\eta(J))$, if $\Delta$ is
considered as a multi-fan on $[m']$. Volume polynomial is
additive, thus we get

\begin{equation}\label{eqSumElemMFans}
V_{\Delta}=\sum_{J\in {[m']\choose
n+1}}z(J)V_{\Delta^{el}(\eta(J))}.
\end{equation}

Therefore, any simplicial chain whose boundary is $w_{ch}$ gives a
formula for the volume polynomial. Now let us consider the
particular case, namely, the cone over $w_{ch}$. Let
$[m']=[m]\sqcup\{r\}$ and set $\eta(r)=v$, for a generic vector
$v\in V$. So the phrase ``$v$ is a generic vector'' means that the
set $\lambda(I)\sqcup\{v\}$ is in general position for any
$I\subset[m]$ such that $|I|=n$ and $w(I)\neq 0$. The function $z$
on the cone is defined in an obvious way: $z(I\sqcup\{r\}):=w(I)$.

Relation \eqref{eqSumElemMFans} and Lemma
\ref{lemElementaryMultifan} imply
\begin{equation}\label{eqVolPolyConst}
V_\Delta=\sum_{I=\{i_1,\ldots,i_n\}\subset [m]}\const\cdot
w(I)\cdot(\alpha_{I,1}c_{i_1}+\cdots+\alpha_{I,n}c_{i_n}+\beta_Ic_r)^n.
\end{equation}
The tuple $(\alpha_{I,1},\ldots,\alpha_{I,n},\beta_I)$ is a linear
relation on the vectors $\lambda(i_1),\ldots,\lambda(i_n),v$.
Therefore we may assume that $\beta_I=-1$ and
$(\alpha_{I,1},\ldots,\alpha_{I,n})$ are the coordinates of $v$ in
the basis $\lambda(i_1),\ldots,\lambda(i_n)$.

Left hand side of \eqref{eqVolPolyConst} does not depend on $c_r$
(it is a redundant support parameter), therefore we may put
$c_r=0$:
\begin{equation}\label{eqVolPolyConst2}
V_\Delta=\sum_{I=\{i_1,\ldots,i_n\}\subset [m]}A_I\cdot
w(I)\cdot(\alpha_{I,1}c_{i_1}+\cdots+\alpha_{I,n}c_{i_n})^n.
\end{equation}

To compute the constants $A_I$ take any $J=\{j_1,\ldots,j_n\}\in
K$ and apply the differential operator $\dd_J=\frac{\dd}{\dd
c_{j_1}}\cdot\ldots\cdot \frac{\dd}{\dd c_{j_n}}$ to the identity
\eqref{eqVolPolyConst2}. On the left we have
$\dd_JV_{\Delta}=\frac{w(J)}{|\det\lambda_J|}$, according to Lemma
\ref{lemDerivsOfVolPol}. On the right side all summands with
$I\neq J$ vanish, and the one with $I=J$ contributes $n!\cdot
A_J\cdot w(J)\prod_i\alpha_{J,i}$. Thus
$A_J=\frac{1}{n!}\frac{1}{|\det\lambda_J|\cdot\prod_i\alpha_{J,i}}$
and the statement follows.
\end{proof}

\begin{rem}
Note that the formula \eqref{eqVolPolyFormula} can be applied to
compute the volume of a simple convex polytope in the case when
the polytope is described as the intersection of half-spaces with
the given equations. In this case the formula is known as
Lawrence's formula \cite{Law}. It has found applications in
explicit volumes' calculations.
\end{rem}

\begin{ex}\label{exCPgeneric}
Consider the standard fan $\Delta$ of $\Co P^2$, generated by the
vectors $\lambda(1)=(1,0)$, $\lambda(2)=(0,1)$,
$\lambda(3)=(-1,-1)$. Take the generic vector $v=(1,2)$. We have
\[
v=\lambda(1)+2\lambda(2)=\lambda(2)-\lambda(3)=-2\lambda(3)-\lambda(1).
\]
Theorem \ref{thmVolPolyFormula} implies
\[
V_{\Delta}=\frac12\left(\frac12(c_1+2c_2)^2 - (c_2-c_3)^2 +
\frac12(-c_1-2c_3)^2\right).
\]
This expression equals $\frac12(c_1+c_2+c_3)^2$. The same
expression is given by Lemma~\ref{lemElementaryMultifan}.
\end{ex}

\begin{ex}\label{exDiscretePolariz}
Consider the normal fan of the standard $n$-cube. The underlying
simplicial complex is isomorphic to the boundary of
cross-polytope. Let $\{1,\ldots,n,-1,\ldots,-n\}$ be its set of
vertices, so the maximal simplices have the form
$\{\pm1,\ldots,\pm n\}$. We have $\lambda(\pm i)=\pm e_i$. Take
the generic vector $v=e_1+\cdots+e_n$. Then Theorem
\ref{thmVolPolyFormula} implies
\[
V_{\Delta}=\frac{1}{n!}\sum_{(\epsilon_1,\ldots,\epsilon_n)\in
\{+,-\}^n}\frac{1}{\prod_{i=1}^n\epsilon_i}
(\epsilon_1c_{\epsilon_11}+\cdots+\epsilon_nc_{\epsilon_nn})^n.
\]
On the other hand, we have $V_{\Delta}=\prod_i(c_i+c_{-i})$ by
geometrical reasons. Indeed, the polytope dual to $\Delta$ is the
brick with sides $\{c_i+c_{-i}\}_{i\in [n]}$. By setting
$c_{-i}=0$ for each $i$ we get the identity
\begin{equation}\label{eqPolarizDiscr}
\prod_{i=1}^nc_i =
\frac{1}{n!}\sum_{I\subseteq[n]}(-1)^{n-|I|}c_I^n,
\end{equation}
where $c_I=\sum_{i\in I}c_i$. This identity is well known as
discrete polarization identity.
\end{ex}

\begin{rem}
The proof of Theorem \ref{thmVolPolyFormula} implies the following
consideration. Take two simplicial $n$-chains $z_{ch,1},
z_{ch,2}\in C_{n}(\triangle_{[m]};\Ro)$ endowed with functions
$\eta_1,\eta_2\colon [m]\to \Ro^n$ such that $\eta_{\epsilon}(J)$
is in general position for any simplex $J$ of the chain
$z_{ch,\epsilon}$, $\epsilon=1,2$. Assume that $d z_{ch,1}=d
z_{ch, 2}$ and the functions $\eta_1,\eta_2$ agree on the vertices
of the boundary. Then the volume polynomial of the multi-fan
$\Delta=(d z_{ch,1},\eta_1)=(d z_{ch,2},\eta_2)$ can be expressed
by two formulas:
\begin{multline*}
\sum_{J=(j_1,\ldots,j_{n+1})\subset [m']}\const\cdot
z_1(J)(\alpha_{J,1}c_{j_1}+\cdots+\alpha_{J,n+1}c_{j_{n+1}})^n=
V_\Delta\\ =\sum_{J=(j_1,\ldots,j_{n+1})\subset [m']}\const\cdot
z_2(J)(\alpha_{J,1}c_{j_1}+\cdots+\alpha_{J,n+1}c_{j_{n+1}})^n.
\end{multline*}
We may take a difference of the left and right parts and summarize
as follows. Let us take any closed simplicial $n$-chain $z_{ch}$,
$dz_{ch}=0$, on the vertex set $[m']$, and endow it with a
function $\eta\colon[m']\to\Ro^n$ which is in general position on
any simplex $J$ of the chain. Then we get an identity
\[
\sum_{J=(j_1,\ldots,j_{n+1})\subset [m']}\const\cdot
z(J)(\alpha_{J,1}c_{j_1}+\cdots+\alpha_{J,n+1}c_{j_{n+1}})^n=0
\]
(the constants may be computed by the same method as we used
previously). This seems to be a quite general way to construct
algebraical identities from geometrical data.

This idea can be illustrated by a simple identity obtained in
Example \ref{exCPgeneric}:
\[
\frac12(c_1+2c_2-c_4)^2 - (c_2-c_3-c_4)^2 +
\frac12(-c_1-2c_3-c_4)^2 = (c_1+c_2+c_3)^2.
\]
This identity is induced by the schematic picture shown on
Fig.\ref{figSimpChains}.
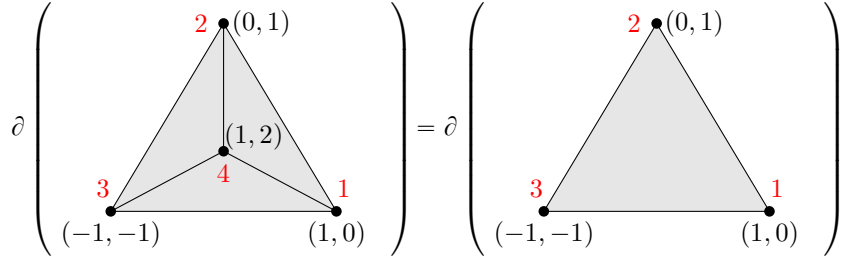
\begin{figure}[h]
\[\dd\left(
\parbox[c]{4.5cm}{
    \begin{tikzpicture}[scale=1]
        \filldraw[fill=black!10]
        (0,0)--(3,0)--(1.5,2.5)--cycle;
        \draw (0,0)--(1.5,0.8);
        \draw (3,0)--(1.5,0.8);
        \draw (1.5,2.5)--(1.5,0.8);
        \fill (0,0) circle(2pt);
        \fill (3,0) circle(2pt);
        \fill (1.5,2.5) circle(2pt);
        \fill (1.5,0.8) circle(2pt);
        \draw[black] (0,-0.3) node{$(-1,-1)$};
        \draw[black] (3,-0.3) node{$(1,0)$};
        \draw[black] (2,2.5) node{$(0,1)$};
        \draw[black] (1.9,1) node{$(1,2)$};

        \draw[red] (-0.1,0.3) node{$3$};
        \draw[red] (3.1, 0.3) node{$1$};
        \draw[red] (1.2,2.5) node{$2$};
        \draw[red] (1.5,0.5) node{$4$};

    \end{tikzpicture}}\right)
=\dd\left(\parbox[c]{4.5cm}{
    \begin{tikzpicture}[scale=1]
        \filldraw[fill=black!10]
        (0,0)--(3,0)--(1.5,2.5)--cycle;
        \fill (0,0) circle(2pt);
        \fill (3,0) circle(2pt);
        \fill (1.5,2.5) circle(2pt);
        \draw[black] (0,-0.3) node{$(-1,-1)$};
        \draw[black] (3,-0.3) node{$(1,0)$};
        \draw[black] (2,2.5) node{$(0,1)$};

        \draw[red] (-0.1,0.3) node{$3$};
        \draw[red] (3.1, 0.3) node{$1$};
        \draw[red] (1.2,2.5) node{$2$};
    \end{tikzpicture}}\right)
\]
\caption{Two simplicial chains with vector functions having the
same boundary} \label{figSimpChains}
\end{figure}
Note that the last step in the proof of Theorem
\ref{thmVolPolyFormula} was to specialize $c_4=0$, but even
without this specialization the identity holds true.
\end{rem}

%
%
%
%
%
%

\section{Poincare duality algebra of a
multi-fan}\label{secPDAmultifan}

%
%

\subsection{Poincare duality algebras}\label{subsecPDAgeneral}

\begin{defin}
Let $\ko$ be a field. Let $\A^*=\bigoplus_{j=0}^n\A^{2j}$ be a
finite-dimensional graded commutative $\ko$-algebra such that
\begin{itemize}
\item there exists an isomorphism $\int_{\A}\colon \A^{2n}\to \ko$;
\item the pairing $\A^{2p}\otimes \A^{2n-2p}\to \ko$, $a\otimes b\mapsto \int_{\A} (a\cdot
b)$ is non-degenerate.
\end{itemize}
Then $\A$ is called \emph{a Poincare duality algebra} of formal
dimension $2n$.
\end{defin}

Let $\dd_i=\frac{\dd}{\dd c_i}$, $i\in[m]$ be the differential
operator acting on the ring of polynomials $\Ro[c_1,\ldots,c_m]$
in a standard way. For a subset $I\subset [m]$ let $\dd_I$ denote
the product $\prod_{i\in I}\dd_i$.

Consider the algebra of differential operators with constant
coefficients $\D:=\Ro[\dd_1,\ldots,\dd_m]$. It will be convenient
to double the degree, so we assume $\deg\dd_i=2$, $i\in[m]$ (while
still assuming that $\deg c_i=1$). For any non-zero homogeneous
polynomial $\Psi\in\Ro[c_1,\ldots,c_m]$ of degree $n$ we may
consider the following ideal in $\D$:
\[
\Ann\Psi:=\{D\in\D\mid D\Psi=0\}.
\]
It is not difficult to check that the quotient $\D/\Ann\Psi$ is a
Poincare duality algebra of formal dimension $2n$ (see
\cite[Prop.2.5.1]{Tim}), where the ``integration map'' assigns the
number $D\Psi\in \Ro$ to any differential operator of rank $n$
(i.e. of formal degree $2n$ in our setting).

It happens that every Poincare duality algebra generated by degree
two can be obtained by this construction as the following
proposition shows.

\begin{prop}\label{propPDdescription}
Suppose $\ch\ko=0$ and let $\ko[m]=\ko[x_1,\ldots,x_m]$ be a
polynomial ring, where $\deg x_i=2$. Then the following three sets
of objects are naturally equivalent:
\begin{enumerate}
\item Poincare duality algebras $\A^*$ of formal dimension $2n$ which are
the quotients of the polynomial ring $\ko[m]$;

\item Non-zero homogeneous polynomials $\Psi\in
\ko[c_1,\ldots,c_m]$ of degree $n$ (where $\deg c_i=1$) up to
multiplication by a non-zero constant;

\item Non-zero linear maps $\int\colon
\ko[m]_{2n}\to\ko$ up to multiplication by a non-zero constant.
\end{enumerate}
\end{prop}

\begin{proof}
We give a very brief sketch of the proof. For details the reader
is referred to the monograph \cite{MS} which, among other things,
describes the case $\ch\ko\neq0$ (for general fields instead of a
polynomial $\Psi$ one should take an element of divided power
algebra). Also we would like to mention that an equivalence of (1)
and (2) is a manifestation of the well-known phenomenon called
Macaulay duality (or its extended version, Matlis duality).

(1)$\Rightarrow$(3). Let $\A^*\cong\ko[m]/\ca{I}$ be a Poincare
duality quotient of the ring of polynomials. Then we have a linear
isomorphism $\int_{\A}\colon \A^{2n}\to \ko$. The composite
\[
\ko[m]_{2n}\twoheadrightarrow \A^{2n}\to \ko
\]
is the required linear map.

(3)$\Rightarrow$(1). Given a linear map $\int\colon\ko[m]_{2n}\to
\ko$ we may define a pairing $\ko[m]_{2p}\otimes\ko[m]_{2n-2p}\to
\ko$ by $a\otimes b\mapsto \int a\cdot b$. This pairing is
degenerate and we define its kernel:
\[
W^*=\bigoplus\nolimits_p W^{2p},\quad W^{2p}=\{x\in
\ko[m]_{2p}\mid \int x\cdot \ko[m]_{2n-2p}=0\}.
\]
It is easy to check that $W^*\subset \ko[m]$ is an ideal and
$\ko[m]/W^*$ is a Poincare duality algebra.

(3)$\Rightarrow$(2). We construct a polynomial $\Psi$ in
$c_1,\ldots,c_m$ by
\[
\Psi:=\frac{1}{n!}\int(c_1x_1+\cdots+c_mx_m)^n.
\]
This polynomial is non-zero. Indeed, $\ko[m]_{2n}$ is additively
generated by the monomials of degree $n$ in the variables $x_i$.
Each monomial can be expressed as a linear combination of
expressions of the form $(c_1x_1+\cdots+c_mx_m)^n$ for some
constants $c_i$ as follows from the polarization identity (see
\eqref{eqPolarizDiscr} in Example \ref{exDiscretePolariz} below).
Thus expressions of the form $(c_1x_1+\cdots+c_mx_m)^n$ linearly
span $\ko[m]_{2n}$ and therefore, since $\int$ is non-zero, the
polynomial $\Psi$ is not a constant zero as well.

(2)$\Rightarrow$(1). Given a homogeneous polynomial $\Psi$ in the
variables $c_1,\ldots,c_m$, we may construct a Poincare duality
quotient $\ko[\dd_1,\ldots,\dd_m]/\Ann\Psi$, where the action of
$\dd_i=\frac{\dd}{\dd c_i}$ on polynomials is defined formally in
the usual way.

The consistency of all these constructions is a routine check.
\end{proof}

The same arguments can be used to prove that there is a one-to-one
correspondence between Poincare duality quotients of formal
dimension $2n$ of an algebra $\ca{B}^*$ and the non-zero linear
functionals on the linear space $\ca{B}^{2n}$. For this
correspondence we do not need the assumptions that $\ca{B}$ is
generated by degree $2$ and $\ch\ko=0$. This motivates the
following definition.

\begin{defin}\label{definPoincarization}
Let $\ca{B}^*=\bigoplus_{j}\ca{B}^{2j}$ be a graded commutative
$\ko$-algebra and suppose that for some $n>0$ a non-zero linear
map $\int\colon \ca{B}^{2n}\to\ko$ is given. The corresponding
Poincare duality quotient of $\ca{B}^*$, i.e. the algebra
\[
\ca{B}^*/W^*,\quad W^{2p}=\{b\in \ca{B}^{2p}\mid \int b\cdot
\ca{B}^{2n-2p}=0\}.
\]
is denoted by $\PD(\ca{B}^*,\int)$ and called \emph{Poincare
dualization} of $\ca{B}^*$ (w.r.t. $\int$).
\end{defin}

\begin{lemma}\label{lemPDinduced}
Consider two algebras $\ca{B}_1^*,\ca{B}_2^*$ with the given
non-zero linear maps $\int_1\colon \ca{B}_1^{2n}\to \ko$,
$\int_2\colon\ca{B}_2^{2n}\to\ko$. Let $\varphi\colon
\ca{B}_1^*\twoheadrightarrow\ca{B}_2^*$ be an epimorphism of
algebras consistent with the integration maps:
$\int_2\circ\varphi|_{\ca{B}_1^{2n}}=\int_1$. Then $\varphi$
induces an isomorphism $\PD(\ca{B}_1^*,\int_1)\cong
\PD(\ca{B}_2^*,\int_2)$.
%
\end{lemma}

\begin{proof}
From the surjectivity of $\varphi$ it easily follows that the
kernel $W^*_1$ of the intersection pairing in the first algebra
maps to the kernel $W_2^*$ of the second algebra. Thus the
homomorphism $\widetilde{\varphi}\colon\PD(\ca{B}_1^*,\int_1)\to
\PD(\ca{B}_2^*,\int_2)$ is well defined. Obviously it is
surjective. Let us prove that $\widetilde{\varphi}$ is injective.
The map $\widetilde{\varphi}$ is an isomorphism in degree $2n$.
Suppose that $\widetilde{\varphi}(a)=0$ for some $0\neq a\in
\PD(\ca{B}_1^*,\int_1)_{2p}$. By the definition of Poincare
duality algebra, there exists $b\in
\PD(\ca{B}_1^*,\int_1)_{2n-2p}$ such that $ab\neq 0$. But then we
have
$\widetilde{\varphi}(ab)=\widetilde{\varphi}(a)\widetilde{\varphi}(b)=0$
which gives a contradiction.
\end{proof}

In the following let $\A^*(\Psi)=\D/\Ann\Psi$ denote the Poincare
duality algebra corresponding to the homogeneous polynomial $\Psi$
of degree $n$.

\subsection{Algebras associated with
multi-fans}\label{subsecPDAMFan}

The linear maps $\int_{\Delta}\colon H^{2n}(\Delta)\to \Ro$ and
$\int_{\Delta,\Ro[m]}\colon \Ro[x_1,\ldots,x_m]_{2n}\to \Ro$ are
consistent with the natural projection $\Ro[x_1,\ldots,x_m]\to
H^*(\Delta)$. Thus Lemma \ref{lemPDinduced} implies an isomorphism
\[
\PD(H^*(\Delta),\int_{\Delta})\cong
\PD(\Ro[m],\int_{\Delta,\Ro[m]}).
\]
According to the constructions mentioned in the proof of
Proposition \ref{propPDdescription}, this Poincare duality algebra
is also isomorphic to $\A^*(V_\Delta)=\D/\Ann V_\Delta$, where
$V_\Delta$ is the volume polynomial.

\begin{defin}
Let $\Delta$ be a complete simplicial multi-fan of dimension $n$
with $m$ rays. Then the algebra
\[
\A^*(\Delta):=\D/\Ann V_\Delta\cong
\PD(H^*(\Delta),\int_{\Delta})\cong
\PD(\Ro[m],\int_{\Delta,\Ro[m]})
\]
is called \emph{a multi-fan algebra} of $\Delta$. 
\end{defin}


\begin{rem}\label{remSimplePropMFAlgebra}
The constructions above show that there is a ring epimorphism from
$H^*(\Delta)\cong \Ro[K]/\Theta$ to $\A^*(\Delta)\cong
\PD(H^*(\Delta),\int_{\Delta})$, sending $x_i$ to $\dd_i$ for each
$i\in [m]$. Therefore $\A^*(\Delta)$ can be considered as a
quotient of $H^*(\Delta)$, and all the relations in
$\Ro[K]/\Theta$ are inherited by $\A^*(\Delta)$. We have
\[
\dd_JV_\Delta = 0 \mbox{ for } J\notin K
\qquad\mbox{(Stanley--Reisner relations),}
\]
\[
(\sum_{i\in[m]}\lambda_{i,j}\dd_i)V_\Delta = 0 \mbox{ for
}j=1,\ldots,n \qquad\mbox{(Linear relations).}
\]
This proves points 1 and 2 of Lemma \ref{lemDerivsOfVolPol} in a
more conceptual way.
\end{rem}

%
%
%
%
%
%

\section{Structure of multi-fan algebra in particular
cases}\label{secStructMFanAlg}

%
%
%

\subsection{Ordinary fans}\label{subsecStructFans}

As was mentioned in the introduction, when $\Delta$ is a normal
fan of a simple convex polytope $P$, the construction of the
algebra $\A^*(\Delta)=\D/\Ann V_\Delta$ was introduced by Timorin
in \cite{Tim}. In this case the underlying simplicial complex of
$\Delta$ is a sphere and the weight function takes value $+1$ on
all maximal simplices of $K$. Using purely combinatorial and
geometrical considerations Timorin proved that $\A^*(\Delta)\cong
\Ro[K]/\Theta$. This means, in particular, that the dimension
$d_i=\dim \A^{2i}(\Delta)$ is equal to $h_i$, the $h$-number of
$K$ (see the definition below). The developed technique is applied
to prove that $\A^*(\Delta)$ is a Lefschetz algebra, meaning that
there exists an element $\omega\in \A^2(\Delta)$ such that
\[
\times \omega^{n-2k}\colon \A^{2k}\to \A^{2n-2k}
\]
is an isomorphism for each $k=0,\ldots,[n/2]$. In particular this
implies that the distribution of $h$-numbers of convex simplicial
spheres is unimodal, i.e.
\[
h_0\leqslant h_1\leqslant \cdots\leqslant h_{[n/2]}=
h_{n-[n/2]}\geqslant \cdots\geqslant h_{n-1}\geqslant h_n.
\]
According to Timorin's result, Lefschetz element $\omega$ may be
chosen in the form $c_1(P)=c_1\dd_1+\cdots+c_m\dd_m\in
\A^2(\Delta)$ where $P$ is any convex simple polytope with the
normal fan $\Delta$ and $c_1,\ldots,c_m$ are its support
parameters.

For complete non-singular fans the algebra
$\A^*[\Delta]\cong\Ro[K]/\Theta$ coincides with the cohomology
algebra $H^*(X_\Delta;\Ro)$ of the corresponding toric variety. It
was the original observation of Stanley \cite{StG}, that in the
case when a fan $\Delta$ is polytopal, the corresponding complete
toric variety $X_\Delta$ is projective, therefore there exists a
Lefschetz element in its cohomology ring according to hard
Lefschetz theorem.

After Stanley's work, several approaches were developed to prove
the existence of Lefschetz elements in elementary terms, i.e.
without referring to hard Lefschetz theorem. These approaches
include in particular McMullen's construction of the polytope
algebra \cite{McM}, the approach based on continuous piece-wise
polynomial functions \cite{Brion}, and Timorin's construction
based on the volume polynomial and differential operators
\cite{Tim}.

We will see that ordinary fans are not the only examples of
multi-fans for which the structure of $\A^*(\Delta)$ can be
explicitly described. On the other hand, $\A^*(\Delta)$ is always
a Poincare duality algebra, so it is natural to ask if it is
Lefschetz (or at least if the dimension vector
$(d_0,d_1,\ldots,d_n)$ is unimodal). Later we will show that this
is not true in general, see Theorem \ref{thmEveryAlgIsMF}.

%
%
%

\subsection{Combinatorial preliminaries}\label{subsecStructCombinPrelim}

For now we concentrate on multi-fans based on oriented
pseudomanifolds as described in Example \ref{exMFfromSimpComp}.
Let $K$ be a pure simplicial complex of dimension $n-1$ on the
vertex set $[m]$.

Let $f_j$ denote the number of $j$-dimensional simplices of $K$
for $j=-1,0,\ldots,n-1$, in particular we assume that $f_{-1}=1$
(this reflects the fact that the empty simplex formally has
dimension $-1$). The $h$-numbers of $K$ are defined by the
formula:
\begin{equation}\label{eqHvecDefin}
\sum_{j=0}^nh_jt^{n-j}=\sum_{j=0}^nf_{j-1}(t-1)^{n-j},
\end{equation}
where $t$ is a formal variable. Let $\br_j(K)$ denote the reduced
Betti number $\dim \Hr_j(K)$ of $K$. \emph{The $h'$- and
$h''$-numbers} of $K$ are defined by the formulas
\begin{equation}\label{eqDefHprime}
h_j'=h_j+{n\choose
j}\left(\sum_{s=1}^{j-1}(-1)^{j-s-1}\br_{s-1}(K)\right)\mbox{ for
} 0\leqslant j\leqslant n;
\end{equation}
\begin{equation}\label{eqDefHtwoprimes}
h_j'' = h_j'-{n\choose j}\br_{j-1}(K) = h_j+{n\choose
j}\left(\sum_{s=1}^{j}(-1)^{j-s-1}\br_{s-1}(K)\right)
\end{equation}
for $0\leqslant j\leqslant n-1$, and $h''_n=h'_n$. The sum over an
empty set is assumed zero.

%
%
%

\subsection{Homology spheres}\label{subsecStructSpheres}

\begin{defin}
$K$ is called \emph{Cohen--Macaulay} (over $\ko$), if
$\Hr_j(\lk_KI;\ko)=0$ for any $I\in K$ and $j<\dim\lk_KI=n-1-|I|$.
If, moreover, $\Hr_{n-1-|I|}(\lk_KI;\ko)\cong\ko$ for any $I\in
K$, then $K$ is called \emph{Gorenstein*} or \emph{(generalized)
homology sphere}.
\end{defin}

The famous theorems of Reisner and Stanley (the reader is referred
to the monograph \cite{St}) tell that whenever $K$ is
Cohen--Macaulay (resp. Gorenstein*), its Stanley--Reisner algebra
$\ko[K]$ is Cohen--Macaulay (resp. Gorenstein).

Given a characteristic function $\lambda\colon[m]\to V\cong \Ro^n$
we obtain a linear system of parameters
$\theta_1,\ldots,\theta_n\in\Ro[K]$. It generates an ideal which
we denoted by $\Theta\subset \Ro[K]$ in subsection
\ref{subsecSRrings}. In Cohen--Macaulay case every linear system
of parameters is a regular sequence. This implies \cite{St}:
\[
\dim (\Ro[K]/\Theta)_{2j} = h_j.
\]
If $K$ is a homology sphere, then $\Ro[K]$ is Gorenstein. Thus its
quotient by a linear system of parameters $\Ro[K]/\Theta$ is a
Gorenstein algebra of Krull dimension zero. This implies that
$\Ro[K]/\Theta$ is a Poincare duality algebra \cite[Part 1]{MS}.

Now let $\Delta$ be a complete multi-fan based on a homology
sphere $K$. We have the ring epimorphism
$\Ro[K]/\Theta\to\A^*(\Delta)$ (see Remark
\ref{remSimplePropMFAlgebra}). Since both algebras have Poincare
duality, it is an isomorphism (see Lemma \ref{lemPDinduced}). This
proves the following

\begin{thm}\label{thmStructSphere}
Let $\Delta$ be a complete multi-fan based on a homology sphere
$K$. Then $\A^*(\Delta)\cong\Ro[K]/\Theta$. It follows that $\dim
\A^{2j}(\Delta)=h_j$, the $h$-number of~$K$.
\end{thm}

Note that Poincare duality implies the well-known
\emph{Dehn--Sommerville relations for homology spheres}:
$h_j=h_{n-j}$.

We are in position to prove Lemma \ref{lemElementaryMultifan}
which states that the volume polynomial of an elementary multi-fan
$\Delta$ on the vectors $\lambda(i)\in V$ $(i=1,\ldots,n+1)$, is
equal, up to multiplicative constant, to
$(\sum_{i=1}^{n+1}\alpha_ic_i)^n$, where
$(\alpha_1,\ldots,\alpha_{n+1})$ is a linear relation on
$\lambda(i)$'s.

\begin{proof}[Proof of Lemma \ref{lemElementaryMultifan}]
The underlying simplicial complex of $\Delta$ is the boundary of a
simplex, which is a sphere. Therefore, by Theorem
\ref{thmStructSphere} we have $\A^*(\Delta)\cong \Ro[\dd
\triangle_{[n+1]}]/\Theta$. Hence the ideal $\Ann V_\Delta\subset
\Ro[\dd_1,\ldots,\dd_{n+1}]$ is generated by
$\prod_{i=1}^{n+1}\dd_i$ (Stanley--Reisner relation) and linear
differential operators $\theta_j=\sum_{i=1}^{n+1}
\lambda_{i,j}\dd_i$ for $j=1,\ldots,n$. Here
$(\lambda_{i,j})_{j=1}^n$ are the coordinates of the vector
$\lambda(i)$ for each $i=1,\ldots,n+1$. Since
$\sum_{i=1}^{n+1}\alpha_i\lambda(i)=0$ we have a linear relation
$\sum_{i=1}^{n+1}\alpha_i\lambda_{i,j}=0$ for each $j=1,\ldots,n$.
Now it is easy to check that the differential operators
$\prod_{i=1}^{n+1}\dd_i$ and $\theta_j=\sum_{i=1}^{n+1}
\lambda_{i,j}\dd_i$, $j=1,\ldots,n$ annihilate the polynomial
$(\alpha_1c_1+\cdots+\alpha_{n+1}c_{n+1})^n$. Thus, according to
Proposition \ref{propPDdescription}, $V_\Delta$ coincides with
$(\alpha_1c_1+\cdots+\alpha_{n+1}c_{n+1})^n$ up to constant.
\end{proof}

%
%
%

\subsection{Homology manifolds}\label{subsecStructManif}

\begin{defin}
$K$ is called \emph{Buchsbaum} (over $\ko$), if
$\Hr_j(\lk_KI;\ko)=0$ for any $I\in K$, $I\neq\varnothing$ and
$j<\dim\lk_KI=n-1-|I|$. If, moreover,
$\Hr_{n-1-|I|}(\lk_KI;\ko)\cong\ko$ for any $I\in K$,
$I\neq\varnothing$, then $K$ is called \emph{a homology manifold}.
$K$ is called \emph{an orientable homology manifold} if
$\Hr_{n-1}(K;\Zo)\cong \Zo$.
\end{defin}

The difference from the Cohen--Macaulay case is that there are no
restrictions on the topology of $K=\lk_K\varnothing$ itself.
Similar to Cohen--Macaulay property, the term ``Buchsbaum''
indicates that the corresponding algebra $\ko[K]$ is Buchsbaum
(the result of Schenzel \cite{Sch}). In Buchsbaum case linear
system of parameters is no longer a regular sequence.
Nevertheless, Buchsbaum complexes are extensively studied. First,
Schenzel's theorem \cite{Sch} tells that if $K$ is a Buchsbaum
complex, then
\[
\dim (\ko[K]/\Theta)_{2j} = h'_j
\]
for $j=0,\ldots,n$ and the $h'$-numbers determined by
\eqref{eqDefHprime}. Second, there is a theory of socles of
Buchsbaum complexes introduced by Novik and Swartz \cite{NS} which
we briefly review next.

Let $\ca{M}$ be a module over the graded polynomial ring
$\ko[m]:=\ko[x_1,\ldots,x_m]$. \emph{The socle} of $\ca{M}$ is the
following subspace
\[
\soc\ca{M}:=\{a\in \ca{M}\mid a\cdot\ko[m]_+=0\}.
\]
which is obviously a $\ko[m]$-submodule of $\ca{M}$.

If $K$ is Buchsbaum, then there exists a submodule
$I_{NS}\subset\soc(\ko[K]/\Theta)$ such that
\[
(I_{NS})_{2j}\cong {n\choose j}\Hr^{j-1}(K;\ko),
\]
where the right hand side means the direct sum of ${n\choose j}$
copies of $\Hr^{j-1}(K;\ko)$. Moreover, the result of \cite{NSgor}
tells that whenever $K$ is an orientable connected homology
manifold, then $I_{NS}$ coincides with $\soc(\ko[K]/\Theta)$ and
the quotient
\[
(\ko[K]/\Theta)/I_{NS}^{<2n}
\]
is a Gorenstein algebra (thus Poincare duality algebra). Here
$I_{NS}^{<2n}$ is the part of $I_{NS}$ taken in all degrees except
the top one, $2n$. The definition of $h''$-numbers
\eqref{eqDefHtwoprimes} implies that
\[
\dim ((\ko[K]/\Theta)/I_{NS}^{<2n})_{2j} = h''_j \mbox{ for }
0\leqslant j\leqslant n.
\]

Now let $\Delta$ be a complete multi-fan based on an orientable
connected homology manifold $K$. Recall from Definition
\ref{definPoincarization} that $W^*$ denotes the subspace of
$H^*(\Delta)=\Ro[K]/\Theta$ whose graded components are
\begin{multline}
W^{2j}=\{a\in (\Ro[K]/\Theta)_{2j}\mid \int a\cdot
(\Ro[K]/\Theta)_{2n-2j}=0\} \\ = \{a\in (\Ro[K]/\Theta)_{2j}\mid
\int a\cdot \Ro[m]_{2n-2j}=0\}.
\end{multline}
By definition,
$\A^*(\Delta)=\PD(\Ro[K]/\Theta)=(\Ro[K]/\Theta)/W^*$. The socle
$\soc(\Ro[K]/\Theta)$ lies in $W^*$ in all degrees except the top
one since it is killed by $\Ro[m]_+$. Therefore we have a
well-defined ring epimorphism
\[
(\Ro[K]/\Theta)/I_{NS}^{<2n}\twoheadrightarrow \A^*(\Delta).
\]
Again, since both algebras have Poincare duality, there holds

\begin{thm}\label{thmStructManif}
Let $\Delta$ be a complete multi-fan based on oriented connected
homology manifold $K$. Then
$\A^*(\Delta)\cong(\Ro[K]/\Theta)/I_{NS}^{<2n}$. It follows that
$\dim \A^{2j}(\Delta)=h_j''$, the $h''$-number of $K$.
\end{thm}

In this case Poincare duality implies the well-known generalized
\emph{Dehn--Sommerville relations for oriented homology
manifolds}: $h_j''=h_{n-j}''$ (see \cite{NS} and references
therein).

%
%

\subsection{General situation}\label{subsecStructGeneral}

Let $\Delta$ be an arbitrary complete multi-fan. In the largest
generality we do not have a combinatorial description for the
dimensions of graded components of the multi-fan algebra.

\begin{conj}\label{conjRigidity}
Let $w_{ch}$ be a simplicial cycle, $\lambda\colon [m]\to V$ a
characteristic map, and $\Delta=(w_{ch},\lambda)$ the
corresponding complete multi-fan. The numbers
$d_j=\dim\A^{2j}(\Delta)$ do not depend on $\lambda$.
\end{conj}


%
%
%
%
%
%

\section{Geometry of multi-polytopes and Minkowski
relations}\label{secMinkRels}

Here we give another proof of Theorem \ref{thmStructManif} which
shows the geometrical nature of the elements lying in the socle of
$\Ro[K]/\Theta$ when $K$ is an oriented homology manifold. It
relates on explicit computations in coordinates but reveals an
interesting connection with \emph{the Minkowski type relations},
appearing in convex geometry. Recall the basic Minkowski theorem
on convex polytopes.

\begin{thmNo}[Minkowski]\label{thmMinkow}
(1) (Direct) Let $P$ be a convex full-dimensional polytope in
euclidian space $\Ro^n$. Let $V_1,\ldots,V_m$ be the
$(n-1)$-volumes of facets of $P$ and
$\mathbf{n}_1,\ldots,\mathbf{n}_m$ be the outward unit normal
vectors to facets. Then $\sum_iV_i\mathbf{n}_i=0$ (the Minkowski
relation).

(2) (Inverse). Let $\mathbf{n}_1,\ldots,\mathbf{n}_m$ be the
vectors of unit lengths, spanning $\Ro^n$, and let
$V_1,\ldots,V_m$ be positive numbers satisfying the Minkowski
relation. Then there exists a convex polytope $P$ whose facets
have outward normal vectors $\mathbf{n}_i$ and volumes $V_i$. Such
polytope is unique up to parallel shifts.
\end{thmNo}

Usually only part (2) is called Minkowski theorem, since part (1)
is fairly simple. The direct Minkowski theorem has a
straightforward generalization.

\begin{thm}\label{thmMinkowGener}
Let $\sum_s a_s Q_s$ be a collection of $k$-dimensional
multi-polytopes in euclidian space $\Ro^n$, forming a closed
orientable cycle. Let $\vol(Q_s)$ be the $k$-volume, and $\nu_s\in
\Lambda^{n-k}\Ro^n$ be the unit normal skew form of the
multi-polytope $Q_s$. Then there holds a relation
$\sum_sa_s\vol(Q_s)\nu_s=0$ in $\Lambda^{n-k}\Ro^n$.
\end{thm}

In the next subsection we explain the precise meaning of the terms
used in the statement and give the proof.

%
%

\subsection{Cycles of multi-polytopes}\label{subsecCyclesMPoly}

As before, $V^*\cong \Ro^n$ denotes the ambient affine space of
$n$-dimensional polytopes, coming with fixed orientation. Let
$k\leqslant n$ and let $\Pi$ be an oriented $k$-dimensional affine
subspace of $V^*$. Let $Q$ be a $k$-dimensional multi-polytope in
$\Pi$. Then $Q$ will be called a $k$-dimensional multi-polytope in
$V^*$.

First let $k>0$. Denote by $\GMP_k$ the group (or a vector space
over $\Ro$) freely generated by all $k$-dimensional
multi-polytopes in $V^*$, where we identify the element
$\overline{Q}$ (i.e. $Q$ with reversed orientation of the
underlying subspace) and $-Q$. If $k=0$, the multi-polytope is
just a point with weight. In this case let $\GMP_0$ denote the
group of formal sums of points whose weights sum to zero. Formally
set $\GMP_{-1}=0$. Define the differential $d\colon \GMP_k\to
\GMP_{k-1}$ by setting
\[
dQ:=\sum_{F_i :\mbox{ facet of }Q} F_i,
\]
and extending by linearity. Note that each facet comes with the
canonical orientation: we say that the hyperplane $H_i$ containing
$F_i$ is positively oriented if
\[
(\mbox{a positive basis of } H_i, \lambda(i))
\]
is a positive basis of $V$. Thus the expression above is well
defined.

\begin{defin}
An element $A=\sum_s a_sQ_s\in \GMP_k$ which satisfies $dA=0$ is
called \emph{a cycle of $k$-dimensional multi-polytopes}.
\end{defin}

As in Section \ref{secBasicPropVolPoly}, assume that there is a
fixed inner product in $V$. This allows to define the inner
product on the skew forms. In particular, if $\Pi$ is an oriented
affine $k$-subspace in $V^*\cong V$, we may define its \emph{unit
normal skew} form $\nu_{\Pi}\in \Lambda^{n-k}V$ as the unique
element of $\Lambda^{n-k}\Pi^\bot\cong\Ro$ which corresponds to
the positive orientation of $\Pi^\bot$ and satisfies
$\|\nu_{\Pi}\|=1$. It is easy to see that if $\dim\Pi=n-1$, the
form $\nu_{\Pi}$ is just the positive unit normal vector to $\Pi$.

Let us prove Theorem \ref{thmMinkowGener}.

\begin{proof}
The idea of proof is straightforward and quite similar to the
proof of classical Minkowski theorem: at first we prove the case
$k=n$, then reduce the general case to the case $k=n$ by
projecting $\sum_sa_s\vol(Q_s)\nu_s$ to all possible
$k$-subspaces. Note that the case $n=0$ should be treated
separately, but in this case the statement is trivial.

(1) Suppose $k=n$. Then all multi-polytopes $Q_s$ are
full-dimensional. Their underlying subspaces $\Pi_s$ coincide with
$V$ up to orientation. Without loss of generality assume that all
orientations coincide with that of $V$. Normal skew forms lie in
$\Lambda^0V\cong\Ro$ and are equal to $1$. Hence we need to prove
that $\sum a_s\vol(Q_s)=0$ for any cycle of $n$-multi-polytopes.
Recall the wall-crossing formula \cite[Lemma 5.3]{HM}:

\begin{lemma}\label{lemWallCrossing}
Let $P$ be a multi-polytope and $H=H_i$ be one of the supporting
hyperplanes: $H=H_i$. Let $u_\alpha$ and $u_\beta$ be elements in
$V^*\setminus \bigcup_{i=1}^mH_i$ such that the segment from
$u_\alpha$ to $u_\beta$ intersects the wall $H$ transversely at
$\mu$, and does not intersect any other $H_j\neq H$. Then
\[
\Dh_P(u_\alpha)- \Dh_P(u_\beta) = \sum_{i:H_i=H} \sgn\langle
u_\beta-u_\alpha, \lambda(i)\rangle \Dh_{F_i}(\mu),
\]
where $F_i$ is the facet of $P$, and $\Dh_{F_i}\colon H_i\to \Ro$
is its Duistermaat--Heckman function.
\end{lemma}

Consider a cycle of multi-polytopes $A=\sum_{s=1}^l a_sQ_s$. Let
$\ca{H}$ denote the set of all supporting hyperplanes of all
polytopes $Q_s$, $s=1,\ldots,l$. We have a function
\[
\Dh_A\colon V^*\setminus \bigcup_{H\in\ca{H}}H \to \Ro, \quad
\Dh_A:=\sum_{s=1}^la_s\Dh_{Q_s}.
\]
Let us choose a hyperplane $H\in \ca{H}$ and two points $u_\alpha$
and $u_\beta$ in $V^*\setminus \bigcup_{H\in\ca{H}}H$ such that
the segment from $u_\alpha$ to $u_\beta$ intersects the wall $H$
transversely at $\mu$ and does not intersect any other wall from
$\ca{H}$. Let us sum the differences $\Dh_P(u_\alpha)-
\Dh_P(u_\beta)$ taken with coefficients $a_s$ over all
multi-polytopes $Q_s$ for which $H$ is a supporting hyperplane.
Since $dA=0$, Lemma \ref{lemWallCrossing} implies that this sum is
zero. Obviously, this sum equals $\Dh_A(u_\alpha)-
\Dh_A(u_\beta)$.

This argument shows that crossing of any wall does not change the
value of $\Dh_A$. Therefore, $\Dh_A$ is constant (where it is
defined). Since $\Dh_A$ has compact support, it must be constantly
zero. Thus
\[
\sum a_s\vol(Q_s)=\int_{V^*}\Dh_A=0.
\]

(2) Let us prove the theorem for general $k$. Consider a generic
oriented $k$-subspace $\Pi\subset V^*$ and let
$\nu\in\Lambda^{n-k}V^*$ be its normal skew form. Let
$\Gamma\colon V^*\to \Pi$ be the orthogonal projection. Then the
image of $Q_s$ under $\Gamma$ is a full-dimensional multi-polytope
in $\Pi$, which we denote by $\Gamma(Q_s)$. The sum
$\sum_{s=1}^la_s\Gamma(Q_s)$ is a cycle of $k$-dimensional
multi-polytopes in $\Pi$. Therefore, step (1) implies
\[
\sum_{s=1}^la_s\vol(\Gamma(Q_s))=0.
\]
By the standard property of orthogonal projections we have
\[
\vol(\Gamma(Q_s))=\vol(Q_s)\cdot\langle \nu_s,\nu\rangle.
\]
Hence
\[
\Big\langle \sum_{s=1}^la_s\vol(Q_s)\nu_s, \nu \Big\rangle = 0,
\]
and this holds for any generic skew form $\nu$. Thus
$\sum_{s=1}^la_s\vol(Q_s)\nu_s=0$ which was to be proved.
\end{proof}

%
%

\subsection{Relations in $\A^*(\Delta)$ as Minkowski
relations}\label{subsecMinkRelations}

Let $K$ be an oriented homology $(n-1)$-manifold and $\Delta$ be a
multi-fan based on $K$. Suppose that every simplex $I\in K$ is
oriented somehow. This defines an orientation of each subspace
$H_I=\bigcap_{i\in I}H_i$ (for example, by the rule ``positive
orientation of $H_i$''$\oplus
\lambda(i_1)\oplus\cdots\oplus\lambda(i_k)$ is a positive
orientation of $V$ if $(i_1,\ldots,i_k)$ is a positive order of
vertices of $I$). Recall from Section \ref{secBasicPropVolPoly}
that $\lambda(I)$ denotes the skew form $\bigwedge_{i\in
I}\lambda(i)$ and $\covol(I)=\|\lambda(I)\|$. Consider an
arbitrary skew form $\mu\in \Lambda^kV^*$ and let
\[
\lambda(I)_\mu:=\langle\lambda(I),\mu\rangle.
\]

Let $C^k(K;\Ro)$, $0\leqslant k\leqslant n-1$ denote the group of
cochains on $K$ and $\delta\colon C^k(K;\Ro)\to C^{k+1}(K;\Ro)$ be
the standard cochain differential. We also need to augment the
cochain complex in the top degree, so we formally set
$C^n(K;\Ro):= \Ro$ and let $\delta\colon C^{n-1}(K;\Ro)\to
C^n(K;\Ro)$ be the evaluation of a cochain on the fundamental
chain of $K$.

An element $a\in C^{k-1}(K;\Ro)$, $k\leqslant n$ will be called a
(coaugmented) cocycle if $\delta a=0$. Then, since $K$ is an
oriented manifold, the Poincare dual $\sum_{I:|I|=k}a(I)F_I$ of
$a$ is a cycle of $(n-k)$-dimensional multi-polytopes in $V^*$.
(Notice that in the case $k=n$ we get a formal sum of points whose
weights sum is zero. If we do not require that $a$ is coaugmented,
then we do not get a cycle of $0$-dimensional multi-polytopes).

\begin{prop}
For any coaugmented cocycle $a\in C^{k-1}(K;\Ro)$ and any $\mu\in
\Lambda^kV$ there exists a relation
\[
\sum_{I:|I|=k}a(I)\lambda(I)_\mu\dd_I = 0
\]
in $\A^*(\Delta)\cong (\Ro[K]/\Theta)/I_{NS}^{<2n}$.
\end{prop}

\begin{proof}
Let us apply $\sum_{I:|I|=k}a(I)\lambda(I)_\mu\dd_I$ to the volume
polynomial $V_\Delta$ and evaluate the result at a point
$\bar{c}=(c_1,\ldots,c_m)$:
\[
\sum_{I:|I|=k}a(I)\lambda(I)_\mu\dd_IV_\Delta
|_{\bar{c}}=\sum_{I:|I|=k}a(I)\lambda(I)_\mu\dfrac{\vol(F_I)}{\covol(I)}.
\]
Here we used Lemma \ref{lemDerivsOfVolPol}. Note that the skew
form $\lambda(I)/\covol(I)$ is by definition a unit normal skew
form to the ambient subspace of a multi-polytope $F_I$. Since
$\sum_{I:|I|=k}a(I)F_I$ is a cycle of multi-polytopes, Theorem
\ref{thmMinkowGener} implies
\[
\sum_{I:|I|=k}a(I)\vol(F_I)\dfrac{\lambda(I)}{\covol(I)}=0.
\]
Taking inner product with $\mu$ implies
\[
\sum_{I:|I|=k}a(I)\lambda(I)_\mu\dfrac{\vol(F_I)}{\covol(I)}=0.
\]
Hence the polynomial
$\sum_{I:|I|=k}a(I)\lambda(I)_\mu\dd_IV_\Delta$ evaluates to zero
at any point $\bar{c}$. Therefore it vanishes as a polynomial.
Thus $\sum_{I:|I|=k}a(I)\lambda(I)_\mu\dd_I\in \Ann V_\Delta$
which proves the statement.
\end{proof}

We see that Minkowski theorem allows to construct linear relations
in $\A^*(\Delta)$. Actually these relations exhaust all relations
in $\A^*(\Delta)$. Let us state the result of \cite{AyV3} in terms
of Minkowski relations:

\begin{prop}[\cite{AyV3}]\label{propMinkManifolds}
Let $K$ be an oriented homology manifold.
\begin{enumerate}
\item There is an isomorphism of vector spaces
\[
(\Ro[K]/\Theta)_{2k}\cong \langle x_I\mid I\in K,
|I|=k\rangle/\Big\langle\sum_{I:|I|=k}a(I)\lambda(I)_\mu x_I
\Big\rangle
\]
where $a$ runs over all exact $(k-1)$-cochains on $K$ and $\mu$
runs over $\Lambda^kV$.

\item There is an isomorphism of vector spaces
\[
((\Ro[K]/\Theta)/I_{NS}^{<2n})_{2k}\cong \langle x_I\mid I\in K,
|I|=k\rangle/\Big\langle\sum_{I:|I|=k}a(I)\lambda(I)_\mu x_I
\Big\rangle
\]
where $a$ runs over all coaugmented closed $(k-1)$-cochains on $K$
and $\mu$ runs over $\Lambda^kV$.
\end{enumerate}
\end{prop}

Recall that $(I_{NS})_{2k}\cong {n\choose k}H^{k-1}(K;\Ro)$. From
Proposition \ref{propMinkManifolds} it can be seen that the
difference between the vector spaces $\Ro[K]/\Theta$ and
$(\Ro[K]/\Theta)/I_{NS}^{<2n}$ arises from the difference between
closed cochains on $K$ and exact cochains. This explains how the
cohomology $H^{k-1}(K)$ appears in the description of $I_{NS}$.
The multiple ${n\choose k}$ comes from the choice of the skew form
$\mu\in \Lambda^kV$ on which we project the Minkowski relation.

\begin{problem}\label{problMinkowskiRels}
Let $\Delta$ be a general complete simplicial multi-fan. Is it
true that $\A^*(\Delta)$ is isomorphic, as a vector space, to the
quotient of $\langle x_I\mid I\in K\rangle$ by linear relations
arising from Minkowski relations? What are these Minkowski
relations?
\end{problem}

%
%

\subsection{Inverse Minkowski theorem}\label{subsecInverseMink}

It is tempting to formulate and prove the inverse Minkowski
theorem for multi-polytopes. First, we need to modify the
statement. The original formulation tells that there exists a
convex polytope with the given normal vectors and the volumes of
facets, but it tells nothing about the combinatorics of the
polytope. We may ask a more specific question, namely

\begin{que}\label{questInvMink}
For a given complete simplicial multi-fan $\Delta$ with the rays
generated by unit vectors $\mathbf{n}_1,\ldots,\mathbf{n}_m$, and
a given $m$-tuple of real numbers $V_1,\ldots,V_m$ satisfying
$\sum V_i\mathbf{n}_i=0$, does there exist a multi-polytope based
on $\Delta$ whose facets have $(n-1)$-volumes $V_1,\ldots,V_m$? If
yes, is it unique?
\end{que}

A simple example shows that the answer, even for the question of
existence, may be negative.

\begin{ex}
Let $\Delta$ be the normal fan of a $3$-dimensional cube. $\Delta$
is an ordinary fan supported by a simplicial complex $K$, which is
the boundary of an octahedron. Let $\{1,2,3,-1,-2,-3\}$ be the set
of vertices of $K$ and $\lambda(\pm1)=(\pm1,0,0)$,
$\lambda(\pm2)=(0,\pm1,0)$, $\lambda(\pm3)=(0,0,\pm1)$ be the
generators of the corresponding rays of $\Delta$ (see Example
\ref{exDiscretePolariz}). The multi-polytopes based on $\Delta$
are the bricks with sides parallel to coordinate axes. Minkowski
relations can be written as $\vol(F_{i})=\vol(F_{-i})$ for
$i=1,2,3$. Let us take the numbers $V_{\pm1}=0$,
$V_{\pm2}=V_{\pm3}=1$. These numbers satisfy Minkowski relations,
but we cannot find a brick whose facets have volumes
$V_{\pm1},V_{\pm2},V_{\pm3}$. Indeed, $V_{\pm1}=0$ implies that
one of the sides of a brick has length $0$, but this would imply
that either $V_{\pm2}=0$ or $V_{\pm3}=0$.
\end{ex}

Nevertheless, the answer to Question \ref{questInvMink} is
completely controlled by the multi-fan algebra. Recall that
$\A^*(\Delta)$ may be interpreted as the algebra of differential
operators $\D$ up to $\Ann(V_\Delta)$. Therefore, for every $a\in
\A^{2j}(\Delta)$, there is a well-defined homogeneous polynomial
$aV_\Delta$ of degree $n-j$. In particular, each element $a\in
\A^{2n-2}(\Delta)$ determines a linear homogeneous polynomial
$aV_{\Delta}=V_1c_1+\cdots+V_mc_m\in \Ro[c_1,\ldots,c_m]_1$. This
linear polynomial is annihilated by
$\theta_j=\sum_{i\in[m]}\lambda_{i,j}\dd_j\in \Ann V_{\Delta}$,
$j=1,\ldots,n$, see Lemma \ref{lemDerivsOfVolPol} or Remark
\ref{remSimplePropMFAlgebra}. This means
\begin{equation}\label{eqMinkMFan}
\sum_{i\in [m]} V_i\lambda(i)=0,
\end{equation}
which can be considered as a Minkowski relation. Let $\Mink$
denote the vector space of all $m$-tuples $(V_1,\ldots,V_m)$
satisfying \eqref{eqMinkMFan}. Thus we obtain a natural
monomorphism $\eta\colon \A^{2n-2}(\Delta)\to \Mink$, $a\mapsto
(V_1,\ldots,V_m)$, where $aV_{\Delta}=V_1c_1+\cdots+V_mc_m$.

\begin{thm}
Let $\Delta$ be a complete simplicial multi-fan with
characteristic function $\lambda$ and assume that $|\lambda(i)|=1$
for each $i\in [m]$. Let $\overline{V}=(V_1,\ldots,V_m)\in \Mink$.
Let $P\in\Poly(\Delta)$ be a multi-polytope and
$\dd_P=c_1\dd_1+\cdots+c_m\dd_m\in \A^2(\Delta)$ be its first
Chern class. Then the polytope $P$ has facet volumes
$V_1,\ldots,V_m$ if and only if
$\eta(\dd_P^{n-1})=(n-1)!\overline{V}$.
\end{thm}

\begin{proof}
Assume that $\eta(\dd_P^{n-1})=(n-1)!\overline{V}$. Then
$\dd_P^{n-1}V_{\Delta}=(n-1)!(V_1c_1+\cdots+V_mc_m)$. Hence
$\dd_i\dd_P^{n-1}V_\Delta=(n-1)!V_i$ for $i\in [m]$. On the other
hand, Corollary~\ref{corDerivsDerivs} implies
\[
\dd_i\dd_P^{n-1}V_\Delta =
(n-1)!\vol(F_i)/\covol(i)=(n-1)!\vol(F_i)/|\lambda(i)|=(n-1)!\vol(F_i).
\]
Thus $V_i=\vol(F_i)$. The other direction is proved similarly.
\end{proof}

Note that $\dim\Mink=m-n$.

\begin{cor}\label{corInverseMink}
Existence in Question \ref{questInvMink} holds for a given
multi-fan $\Delta$ and all $m$-tuples $(V_1,\ldots,V_m)\in\Mink$
if and only if the following two conditions hold:
\begin{enumerate}
\item $\dim \A^2(\Delta)=\dim \A^{2n-2}(\Delta)=m-n$;
\item the power map $\A^2(\Delta)\to \A^{2n-2}(\Delta)$, $\dd\mapsto
\dd^{n-1}$ is surjective.
\end{enumerate}
Uniqueness holds if the power map is bijective.
\end{cor}

\begin{rem}
Note that even the condition $\dim \A^2(\Delta)=m-n$ may fail to
hold. As an example, consider a multi-fan having a ghost vertex,
say $1$. As in general, we have $n$ relations
$\theta_1,\ldots,\theta_n$, lying in the kernel of the linear map
$\langle\dd_1,\ldots,\dd_m\rangle\twoheadrightarrow\A^2(\Delta)$.
But the element $\dd_1$, corresponding to the ghost vertex, also
vanishes in $\A^2(\Delta)$. Thus $\dim \A^2(\Delta)<m-n$.

There exist more nontrivial examples. For example, if the
underlying simplicial complex $K$ is disconnected, with connected
components $K_1,\ldots,K_r$ on disjoint vertex sets
$[m_1],\ldots,[m_r]$, $r>1$, then each connected component
contributes at most $m_1-n$ in the total dimension of
$\A^2(\Delta)$ (see the operation of connected sum of Poincare
duality algebras introduced in subsection \ref{subsecConSums}).
Thus in the disconnected case $\dim\A^2(\Delta)\leqslant m-rn$.
Nevertheless, the inverse Minkowski theorem can be refined in an
obvious way: we should consider Minkowski relations on each
connected component.
\end{rem}

\begin{rem}
The power map $\A^2(\Delta)\to \A^{2n-2}(\Delta)$ is a polynomial
map of degree $n-1$ between real vector spaces of equal
dimensions. It is a complicated object which may be interesting on
its own. One of the consequences from Corollary
\ref{corInverseMink} is that the existence in the inverse
Minkowski theorem holds for a multi-fan $\Delta$ whenever $\dim
\A^2(\Delta)=m-n$ and $n$ is even.
\end{rem}

%
%
%
%
%
%

\section{Recognizing volume polynomials and multi-fan
algebras}\label{secRecognizing}

A natural question is: which homogeneous polynomials are the
volume polynomials, and which Poincare duality algebras appear as
$\A^*(\Delta)$? The answer to the second question seems quite
unexpected.

\begin{thm}\label{thmEveryAlgIsMF}
For every Poincare duality algebra $\A^*$ generated in degree $2$
there exists a complete simplicial multi-fan $\Delta$ such that
$\A^*\cong \A^*(\Delta)$.
\end{thm}

Recall that the symmetric array of nonnegative integers
$(d_0,d_1,\ldots,d_n)$, $d_j=d_{n-j}$, is called unimodal, if
\[
d_0\leqslant d_1\leqslant\cdots\leqslant d_{\lfloor n/2\rfloor}.
\]

\begin{cor}
There exist multi-fans $\Delta$, for which the array
\[
(\dim \A^0(\Delta),\dim \A^2(\Delta),\ldots, \dim \A^{2n}(\Delta))
\]
is not unimodal.
\end{cor}

\begin{proof}
An example of Poincare duality algebra generated in degree $2$,
for which dimensions of graded components are not unimodal was
given by Stanley in \cite{StCE}. Theorem \ref{thmEveryAlgIsMF}
implies that there exists a multi-fan, which produces this
algebra.
\end{proof}

The construction of the volume polynomial is additive with respect
to weights. Let $\MultiFans_\lambda$ denote the vector space of
all multi-fans with the given characteristic function
$\lambda\colon [m]\to V$. Then we obtain a linear map
\begin{equation}\label{eqOmegaDefin}
\Omega_\lambda\colon \MultiFans_\lambda\to \Ro[c_1,\ldots,c_m]_n,
\end{equation}
which maps $\Delta$ to its volume polynomial $V_\Delta$.

Before giving the proof of Theorem \ref{thmEveryAlgIsMF} we
characterize volume polynomials of general position, in the sense
explained below. For this goal we study the properties of the map
$\Omega_\lambda$.

%
%

\subsection{Characterization of volume
polynomials in general position}\label{subsecCharVolPoly}

There is a necessary condition on $V_\Delta$. If $\lambda$ is a
characteristic function and $\theta_j =
\sum_{i\in[m]}\lambda_{i,j}\dd_i \in
\langle\dd_1,\ldots,\dd_m\rangle$, $j=1,\ldots,n$ are the
corresponding linear forms, then $\theta_jV_{\Delta}=0$, see
Remark \ref{remSimplePropMFAlgebra}. Thus the subspace
\[
\Ann^{2}V_{\Delta}=\{D\in \langle\dd_1,\ldots,\dd_m\rangle\mid
DV_\Delta=0\}
\]
has dimension at least $n$. It happens that in
most situations this is also a sufficient condition for a
polynomial to be a volume polynomial.

At first let us consider the situation of general position to
demonstrate the argument. Assume that all characteristic vectors
$\lambda(1),\ldots,\lambda(m)\in \Ro^n$ are \emph{in general
position}, which means that every $n$ of them are linearly
independent. Given a fixed characteristic function $\lambda\colon
[m]\to V$ in general position, we may pick up any simplicial cycle
$w_{ch}\in Z(\triangle_{[m]}^{(n-1)};\Ro)$, consider a complete
multi-fan $\Delta = (w, \lambda)$ and take its volume polynomial.
This defines a map which we previously denoted
by~$\Omega_\lambda$:
\[
\Omega_\lambda\colon
Z(\triangle_{[m]}^{(n-1)};\Ro)\to\Ro[c_1,\ldots,c_m]_n,
\]
from the $(n-1)$-simplicial cycles on $m$ vertices to homogeneous
polynomials of degree $n$. This map is linear and injective by
Corollary \ref{corOmegaInjective}. As before, let $\theta_j$,
$j=1,\ldots,n$ be the linear differential operators associated
with $\lambda$ (i.e. a basis of the image of the map
$\lambda^\top\colon V^*\to (\Ro^m)^*$). Let
\[
\Ann^{n}\Theta=\{\Psi\in \Ro[c_1,\ldots,c_m]_n\mid
\theta_j\Psi=0\mbox{ for each }j=1,\ldots,n\}
\]
be the vector subspace of polynomials annihilated by differential
operators $\Theta=(\theta_1,\ldots,\theta_n)$. As we have seen, if
$\Delta$ has a characteristic function $\lambda$, then
$V_\Delta\in\Ann^n\Theta$. Thus the image of $\Omega_\lambda$ lies
in $\Ann^n\Theta$.

\begin{lemma}\label{lemGenPosit}
If $\lambda$ is in general position, then $\Omega_\lambda\colon
Z_{n-1}(\triangle_{[m]}^{(n-1)};\Ro)\to \Ann^n\Theta$ is an
isomorphism.
\end{lemma}

\begin{proof}
Let us compute the dimensions of domain and target. There are no
$n$-simplices in $\triangle_{[m]}^{(n-1)}$, thus
$Z_{n-1}(\triangle_{[m]}^{(n-1)})=H_{n-1}(\triangle_{[m]}^{(n-1)})$.
All Betti numbers of $\triangle_{[m]}^{(n-1)}$ between the top and
the bottom vanish, thus via Euler characteristic we get
\begin{equation}\label{eqCyclesSkeletonSimplex}
\dim H_{n-1}(\triangle_{[m]}^{(n-1)})= {m\choose n}-{m\choose
n-1}+{m\choose n-2} -\cdots+(-1)^n{m\choose 0}.
\end{equation}

Now let us compute $\dim \Ann^n\Theta$. Take a linear change of
variables $c_1,\ldots,c_m\rightsquigarrow c_1',\ldots,c_m'$ such
that $\theta_j$ becomes the partial derivative $\frac{\dd}{\dd
c_j'}$ for $j=1,\ldots,n$. Thus, after the change of variables,
$\Ann^n\Theta$ becomes the set $\{\Psi\in
\Ro[c_1',\ldots,c_m']_n\mid \frac{\dd}{\dd c_j'}\Psi=0,
j=1,\ldots,n\}$ which is the same as
$\Ro[c'_{n+1},\ldots,c'_m]_n$. Thus
$\dim\Ann^n\Theta={m-n+n-1\choose n}={m-1\choose n}$. This number
coincides with \eqref{eqCyclesSkeletonSimplex}.
\end{proof}

Let $G_{m,n}$ denote the Grassmann manifold of all (unoriented)
$n$-planes in $(\Ro^m)^*$. We can introduce the standard
Pl\"{u}cker coordinates on $G_{m,n}$. If
\begin{equation}\label{eqBasisOfL}
\Big\{\theta_j=\sum_{i=1}^m\lambda_{i,j}x_i\Big\}_{j=1,\ldots,n}
\end{equation}
is a basis in $L\in G_{m,n}$, then the Pl\"{u}cker coordinates of
$L$ are all maximal minors of the $m\times n$ matrix
$(\lambda_{i,j})$.

Any $n$-plane $L\in G_{m,n}$ determines an $m$-tuple of vectors in
$V\cong \Ro^n$ as follows: the basis \eqref{eqBasisOfL} determines
the tuple
$\{\lambda(i)=(\lambda_{i,1},\ldots,\lambda_{i,n})\}_{i\in[m]}$.
The base change in $L$ induces the natural action of $\GL(n,\Ro)$
on the $m$-tuples. By abuse of terminology we call
$\lambda\colon[m]\to V$ the characteristic function corresponding
to $L\in G_{m,n}$ although this characteristic function is
determined only up to automorphism of $V$.

\begin{prop}\label{propVolPolyGenPos}
Let $\Psi\in \Ro[c_1,\ldots,c_m]_n$ be a homogeneous polynomial.
Suppose that the vector subspace $\Ann^2\Psi = \{D\in
\langle\dd_1,\ldots,\dd_m\rangle\mid D\Psi=0\}$ contains an
$n$-dimensional subspace $L\in G_{m,n}$ with all Pl\"{u}cker
coordinates non-zero. Then $\Psi$ is a volume polynomial of some
multi-fan.
\end{prop}

\begin{proof}
Let us pick a basis
$\{\theta_j=\sum_i\lambda_{i,j}x_i\}_{j=1,\ldots,n}$ in $L$
arbitrarily. Non-vanishing of all Pl\"{u}cker coordinates means
that the corresponding characteristic function $\lambda$ is in
general position. By assumption, $\Psi\in \Ann^n\Theta$. Thus
$\Psi$ is a volume polynomial of some multi-fan based on $\lambda$
according to Lemma \ref{lemGenPosit}.
\end{proof}

%
%

\subsection{Proof of Theorem
\ref{thmEveryAlgIsMF}}\label{subsecCharMFanAlgProof}

Let $\A^*$ be an arbitrary Poincare duality algebra over $\Ro$
generated by $\A^2$. Let $2n$ be the formal dimension of $\A$ and
$p = \dim \A^2$. Take any $p+n$ elements $x_1,\ldots,x_{p+n}\in
\A^2$ in general position (i.e. every $p$ of them are linearly
independent). There are $n$ linear relations on
$x_1,\ldots,x_{p+n}$ in $\A^2$ of the form
$\sum_i\lambda_{i,j}x_i=0$, $j=1,\ldots,n$. Since $x_i$ are in
general position, every maximal minor of the $(p+n)\times
n$-matrix $|\lambda_{i,j}|$ is non-zero.

As in the proof of Proposition \ref{propPDdescription}, consider
the polynomial
\[
\Psi_{\A}(c_1,\ldots,c_m) = \frac{1}{n!}\int
(c_1x_1+\cdots+c_mx_m)^n,
\]
where $\int\colon \A^{2n}\stackrel{\cong}{\rightarrow}\Ro$ is any
isomorphism. The linear differential operator
$\theta_j=\sum_i\lambda_{i,j}\dd_i$ annihilates $\Psi_{\A}$ for
$j=1,\ldots,n$. Indeed:
\[
\Big(\sum_{i=1}^m\lambda_{i,j}\frac{\dd}{\dd
c_i}\Big)\frac{1}{n!}\int (c_1x_1+\cdots+c_mx_m)^n =
\frac{1}{n!}\cdot
n\int\Big(\sum_{i=1}^m\lambda_{i,j}x_i\Big)(c_1x_1+\cdots+c_mx_m)^{n-1}=0.
\]
Since $\theta_j$ are in general position, Proposition
\ref{propVolPolyGenPos} implies that $\Psi_{\A}=V_\Delta$ for some
multi-fan $\Delta$. Therefore the corresponding Poincare duality
algebras $\A^*$ and $\A^*(\Delta)$ are isomorphic by Proposition
\ref{propPDdescription}.

%
%

\subsection{Non-general
position}\label{subsecCharVolPolyNonGeneral}

Now we want to study which polynomials are volume polynomials
without the assumption of general position.

Let $I\subset [m]$ and let $\alpha_I\colon \Ro^I\to\Ro^m$ be the
inclusion of the coordinate subspace. Then $\alpha_I^*\colon
(\Ro^m)^*=\langle\dd_1,\ldots,\dd_m\rangle\to (\Ro^I)^*$ is a
projection map. For a linear subspace $\Pi\subset(\Ro^m)^*$ of
dimension at least $n$ consider the following collection of
subsets of $[m]$:
\[
\dep(\Pi):=\{I\subset[m]\mid |I|\leqslant n \mbox{ and }
\alpha_I^*|_\Pi\colon \Pi\to (\Ro^I)^*\mbox{ is not surjective}\}.
\]
\begin{lemma}\label{lemPlaneGenPos}
Let $\Pi\subset (\Ro^m)^*$ and $\dim \Pi\geqslant n$. Then there
exists a subspace $L\subset \Pi$ such that $\dim L=n$ and
$\dep(L)=\dep(\Pi)$.
\end{lemma}

\begin{proof}
When $\dim \Pi=n$ the statement is trivial so we assume
$\dim\Pi>n$. The proof follows from the general position argument.
If $\varphi\colon \Pi\to U$ is an epimorphism, and $\dim
\Pi>n\geqslant \dim U$, then the set of all $n$-planes in $\Pi$
which map surjectively to $U$ is a complement to a subvariety of
positive codimension inside the set of all $n$-subspaces of $\Pi$.
This argument applied to all maps $\alpha_I^*|_\Pi\colon \Pi\to
(\Ro^I)^*$ proves that any generic $n$-plane $L$ in $\Pi$
satisfies $\dep(L)=\dep(\Pi)$.
\end{proof}

Let $\Psi$ be a homogeneous polynomial of degree $n$ and
$\Ann^2\Psi\subset \langle\dd_1,\ldots,\dd_m\rangle=(\Ro^m)^*$ be
its annihilator subspace.

\begin{thm}\label{thmVolPolyChar}
A homogeneous polynomial $\Psi\in \Ro[c_1,\ldots,c_m]_n$ is a
volume polynomial of some complete simplicial multi-fan if and
only if the following conditions hold:
\begin{enumerate}
\item $\dim \Ann^2\Psi\geqslant n$,
\item $\dd_I\Psi=0$ whenever $I\in \dep(\Ann^2\Psi)$.
\end{enumerate}
\end{thm}

\begin{proof}
The necessity of these conditions is already proved. Indeed, the
first condition follows from the fact that $\Ann^2V_\Delta$
contains the image of $\lambda^\top\colon V^*\to
(\Ro^m)^*=\langle\dd_1,\ldots,\dd_m\rangle$ which has dimension
$n$, see Remark \ref{remSimplePropMFAlgebra}. If $I\in
\dep(\Ann^2V_\Delta)$, then $*$-condition (see subsection
\ref{subsecMFansWeightCharFunc}) implies $I\notin K$, and
therefore $\dd_IV_\Delta=0$ by Lemma \ref{lemDerivsOfVolPol}.

Let us prove sufficiency. By Lemma \ref{lemPlaneGenPos} we may
choose an $n$-dimensional plane $L\subset \Ann^2\Psi$ such that
$\dep(L)=\dep(\Ann^2\Psi)$. Therefore, by assumption, $I\in
\dep(L)$ implies $\dd_I\Psi=0$. Let $\lambda\colon [m]\to V$ be
the characteristic function corresponding to $L\in G_{m,n}$. The
condition $I\in\dep(L)$ is equivalent to the condition that
vectors $\{\lambda(i)\}_{i\in I}$ are linearly dependent.

Consider a simplicial complex $\Matr_\lambda$ determined by the
condition: $\{i_1,\ldots,i_k\}\in \Matr_\lambda$ if and only if
$\lambda(i_1),\ldots,\lambda(i_k)$ are linearly independent. Thus
$\Matr_\lambda=2^{[m]}\setminus \dep(L)$. In a sense, the complex
$\Matr_\lambda$ can be considered as a maximal simplicial complex
on $[m]$ for which $\lambda$ is a characteristic function (this
construction is similar to the universal complexes introduced in
\cite{DJ}).

We have $I\notin \Matr_\lambda$ if and only if $\Theta\to
\langle\dd_i\rangle_{i\in I}$ is not surjective. It is easily seen
that multi-fans having characteristic function $\lambda$ are
encoded by the simplicial $(n-1)$-cycles on $\Matr_\lambda$. As
before, we have a map $\Omega_\lambda\colon
Z_{n-1}(\Matr_\lambda;\Ro)\to \Ro[c_1,\ldots,c_m]_n$ which
associates a volume polynomial $V_\Delta$ with a multi-fan
$\Delta=(w_{ch},\lambda)$ for $w_{ch}\in
Z_{n-1}(\Matr_\lambda;\Ro)$. Let
\[
\Ann^n(L,\{\dd_I\}_{I\in \dep(L)})
\]
denote the subspace of all homogeneous polynomials of degree $n$
which are annihilated by linear differential operators from $L$
and by the products $\dd_I$, $I\in \dep(L)$ ($\Leftrightarrow
I\notin \Matr_\lambda$). We already proved that the image of
$\Omega_\lambda$ lies in $\Ann^n(L,\{\dd_I\}_{I\in \dep(L)})$. We
need to prove that the map
\[
\Omega_\lambda\colon Z_{n-1}(\Matr_\lambda;\Ro)\to
\Ann^n(L,\{\dd_I\}_{I\in \dep(L)})
\]
is surjective. Since $\Omega_\lambda$ is injective, it is enough
to show that dimensions of the two spaces are equal.

First of all notice that $\Matr_\lambda$ is by construction the
underlying simplicial complex of a linear matroid. Hence
$\Matr_\lambda$ is a Cohen--Macaulay complex of dimension $n-1$
(see e.g. \cite{St2}). The number $\dim
Z_{n-1}(\Matr_\lambda;\Ro)=\dim \Hr_{n-1}(\Matr_\lambda;\Ro)$ is
called the type of the Cohen--Macaulay complex $\Matr_\lambda$.

Consider the Stanley--Reisner ring
$\Ro[\Matr_\lambda]=\Ro[\dd_1,\ldots,\dd_m]/(\dd_I\mid I\notin
\Matr_\lambda)$, and its quotient by a linear system of parameters
$L\subset \langle\dd_1,\ldots,\dd_m\rangle$:
\[
\Ro[\Matr_\lambda]/(L)=\Ro[\dd_1,\ldots,\dd_m]/(L,\{\dd_I\}_{I\in
\dep(L)})
\]

\begin{claim}
$\dim (\Ro[\Matr_\lambda]/(L))_{2n}=\dim\Ann^n(L,\{\dd_I\}_{I\in
\dep(L)})$.
\end{claim}

This follows from basic linear algebra. There is a non-degenerate
pairing
\[
\Ro[\dd_1,\ldots,\dd_m]_{2n}\otimes\Ro[c_1,\ldots,c_m]_n\to\Ro.
\]
For any subspace $U\subset \Ro[\dd_1,\ldots,\dd_m]_{2n}$ we have
$\dim\Ro[\dd_1,\ldots,\dd_m]_{2n}/U=\dim U^\bot$. Taking the
degree $2n$ part of the ideal $(L,\{\dd_I\}_{I\in \dep(L)})$ as
$U$ proves the claim.

Now, since $\Matr_\lambda$ is Cohen--Macaulay, the socle of
$\Ro[\Matr_\lambda]/(L)$ coincides with
$(\Ro[\Matr_\lambda]/(L))_{2n}$. On the other hand, the dimension
of the socle coincides with the type of Cohen--Macaulay complex
\cite{St}. We have
\[
\begin{split}
\dim\Ann^n(L,\{\dd_I\}_{I\in \dep(L)}) &=
\dim\soc\Ro[\Matr_\lambda]/(L) \\&= \mbox{ type of
}\Matr_\lambda=\dim \Hr_{n-1}(\Matr_\lambda;\Ro)
\end{split}
\]
which finishes the proof of the theorem.
\end{proof}

\begin{rem}
Lemma \ref{lemGenPosit} describing the general position is a
particular case of Theorem \ref{thmVolPolyChar}. In the case of
general position the matroid complex $\Matr_\lambda$ is just the
$(n-1)$-skeleton of a simplex on $m$ vertices.
\end{rem}

%
%
%

\subsection{Global structure of the set of multi-fans}

Let $G_m^n$ denote the Grassmaniann of all codimension $n$ planes
in $\Ro^m\cong\Ro[c_1,\ldots,c_m]_1$. Obviously, $G_{m,n}$ can be
identified with $G_m^n$ by assigning $L^\bot\subset \Ro^m$ to
$L\subset (\Ro^m)^*$. We have already seen, that characteristic
function $\lambda\colon \Ro^m\to V$ determines the element $L\in
G_{m,n}$ defined as the image of $\lambda^\top\colon V^*\to
(\Ro^m)^*$. The corresponding element of $G_m^n$ is the subspace
$Y=L^\bot=\Ker\lambda\subset \Ro^m$.

Let $S^kY$ denote the $k$-th symmetric power of $Y\in G_m^n$, so
we have $S^kY\subset S^k\Ro^m=\Ro[c_1,\ldots,c_m]_k$. We have
\[
S^kY\subset\Ann^kY^\bot=\{\Psi\in\Ro[c_1,\ldots,c_m]_k\mid
D\Psi=0\mbox{ for any }D\in Y^\bot\},
\]
and both spaces have dimension ${m-n+k-1\choose k}$. This implies
that the vector bundle
\[
\{(Y,\Psi)\in G_m^n\times\Ro[c_1,\ldots,c_m]_k\mid \Psi\in
\Ann^kY^\bot\}\to G_m^n
\]
is $S^k\gamma$, the $k$-th symmetric power of the canonical bundle
$\gamma$ over $G_m^n$. We denote its total space by
$E(S^k\gamma)$.

Consider the set of all characteristic functions in $V$ up to
linear automorphism of $V$:
\[
\CharFunc:=\{\lambda\colon\Ro^m\to V\mid \im\lambda=V\}/\GL(V).
\]
Let $\MultiFans$ denote the set of all complete simplicial
multi-fans on the set $[m]$ (considered up to automorphisms of $V$
again). We have a map $\MultiFans\to \CharFunc$ which maps a
multi-fan to its characteristic function. The fiber of this map
over $\lambda$ is the vector space $\MultiFans_\lambda\cong
Z_{n-1}(\Matr_\lambda;\Ro)$ introduced earlier.

We have a commutative square
\[
\xymatrix{ \MultiFans \ar[r] \ar[d] &
E(S^n\gamma) \ar[d]^{S^n\gamma}\\
\CharFunc\ar[r]^(0.55){\cong}& G_m^n}
\]
The lower map associates a codimension $n$ subspace
$Y=\Ker\lambda$ to a characteristic function $\lambda$. The upper
map associates a volume polynomial to a multi-fan. The upper map
is linear on each fiber. The subset of characteristic functions in
general position maps isomorphically to the subset of $G_m^n$ with
non-zero Pl\"{u}cker coordinates; the fiber over a generic point
maps isomorphically according to Lemma \ref{lemGenPosit}.
Exceptional fibers map injectively and their images are described
by Theorem \ref{thmVolPolyChar}.

%
%
%
%
%
%

\section{Surgery of multi-fans and algebras}\label{secSurgery}

In this section we study the behavior of the dimensions
$d_j=\dim\A^{2j}(\Delta)$ under connected sums and flips of
multi-fans.

%
%

\subsection{Connected sums}\label{subsecConSums}

Recall that $\A^*(\Psi)=\Ro[\dd_1,\ldots,\dd_m]/\Ann(\Psi)$
denotes the Poincare duality algebra associated with the
homogeneous polynomial $\Psi$. For a graded algebra (or a graded
vector space) $A^*=\bigoplus_j A^{2j}$ let $\Hilb(A^*;t) =
\sum_j(\dim A^j) t^j$ denote its Hilbert function. Sometimes it
will be convenient to use the notation
$\dm(A^*):=(d_0,d_1,\ldots,d_n)$, where $d_j=\dim A^{2j}$, and
$\dm(\Delta):=\dm(\A^*(\Delta))$.

Let $\A_1\hash \A_2$ denote \emph{a connected sum of two Poincare
duality algebras} of the same formal dimension $2n$. By
definition, $\A_1\hash \A_2=\A_1\oplus \A_2/\sim$ where $\sim$
identifies $\A_1^0$ with $\A_2^0$ and $\A_1^{2n}$ with
$\A_2^{2n}$. Actually, there is an ambiguity in the choice of the
latter identification, so in fact there exists a 1-dimensional
family of connected sums of the given two algebras. We prefer to
ignore this ambiguity in the following (the statements hold for
any representative in the family).

We have $\dm(A_1\hash A_2)=\dm(A_1)+\dm(A_2)-(1,0,\ldots,0,1)$.

\begin{lemma}\label{lemConSumPolys}
Let $\Psi_1\in \Ro[c_1,\ldots,c_m]_n$, $\Psi_2\in
\Ro[c_1',\ldots,c_{m'}']_n$ be the polynomials in distinct
variables. Then
$\A^*(\Psi_1+\Psi_2)\cong\A^*(\Psi_1)\hash\A^*(\Psi_2)$.
\end{lemma}

\begin{proof}
The mixed differential operators $\dd_i\dd_{i'}'$ vanish on
$\Psi_1+\Psi_2$, while $\dd_I(\Psi_1+\Psi_2)$ equals
$\dd_I(\Psi_1)$ (resp. $\dd_I(\Psi_2)$) if $I\subset [m]$ (resp.
$I\subset[m']$).
\end{proof}

Let $\Delta_1$, $\Delta_2$ be two multi-fans, whose vertex sets
are $[m]=\{1,\ldots,n,n+1,\ldots,m\}$ and
$[\widetilde{m}]=\{1,\ldots,n,\widetilde{n+1},\ldots,\widetilde{m}\}$
respectively, and let $I$ denote the set of common vertices:
$I=\{1,\ldots,n\}$. Assume that the weight of $I$ is non-zero in
both multi-fans and assume that characteristic functions of
$\Delta_1$ and $\Delta_2$ coincide on $I$. Then we may consider
$\Delta_1$ and $\Delta_2$ as multi-fans with vertices
$[m]\cup_I[\widetilde{m}]=\{1,\ldots,n,n+1,\ldots,m,\widetilde{n+1},
\ldots, \widetilde{m}\}$ and a common characteristic function. In
this case we call the cone-wise sum $\Delta_1+\Delta_2$ a
connected sum and denote it by $\Delta_1\hash_I\Delta_2$ or simply
$\Delta_1\hash\Delta_2$.

\begin{rem}
It would be natural to assume that $w_1(I)=-w_2(I)$, so that the
cone spanned by $I$ contracts in the connected sum. This is
consistent with the geometrical understanding how ``the connected
sum'' should look like. However, we do not need this assumption in
the following proposition.
\end{rem}

\begin{prop}\label{propConSum}
For a connected sum $\Delta_1\hash\Delta_2$ there holds
\[
\A^*(\Delta_1\hash\Delta_2)\cong\A^*(\Delta_1)\hash\A^*(\Delta_2),
\]
so that
$\dm(\Delta_1\hash\Delta_2)=\dm(\Delta_1)+\dm(\Delta_2)-(1,0,\ldots,0,1)$.
\end{prop}

\begin{proof}
We need a technical lemma
\begin{lemma}
Let $\Delta$ be a multi-fan and $I\subset[m]$ be a vertex set such
that the corresponding characteristic vectors
$\{\lambda(i)\}_{i\in I}$ are linearly independent. Let $V_\Delta$
be the volume polynomial and $V_{\Delta\setminus I}\in \Ro[c_i\mid
i\in [m]\setminus I]_n$ be the homogeneous polynomial obtained by
specializing $c_i=0$ in $V_\Delta$ for each $i\in I$. Then
$\A^*(V_{\Delta\setminus I})\cong \A^*(V_\Delta) (=\A^*(\Delta))$
as Poincare duality algebras.
\end{lemma}

\begin{proof}
Using linear relations $\theta_j=\sum_{i=1}^m\lambda_{i,j}\dd_i=0$
in $\A^2(\Delta)$ we can exclude the variables $\dd_i$ for $i\in
I$. This proves that the set $\{\dd_i\}_{i\in[m]\setminus I}$
spans $\A^2(\Delta)$. Therefore the polynomial
\[
V_{\Delta\setminus
I}=\frac{1}{n!}\int_\Delta\left(\sum\nolimits_{i\in[m]\setminus
I}c_ix_i\right)^n
\]
determines the same Poincare duality algebra as $V_\Delta$.
\end{proof}

By the lemma we have $\A^*(\Delta_1)\cong
\A^*(V_{\Delta_1\setminus I})$ and $\A^*(\Delta_2)\cong
\A^*(V_{\Delta_2\setminus I})$. Polynomials $V_{\Delta_1\setminus
I}$ and $V_{\Delta_2\setminus I}$ have distinct variables, thus
\[
\A^*(V_{\Delta_1\setminus I}+V_{\Delta_2\setminus I})\cong
\A^*(V_{\Delta_1\setminus I})\hash \A^*(V_{\Delta_2\setminus I})
\]
according to Lemma \ref{lemConSumPolys}. It remains to note that
$V_{\Delta_1\setminus I}+V_{\Delta_2\setminus I}$ is the result of
specializing $c_i=0$ for $i\in I$ in the polynomial
$V_{\Delta_1}+V_{\Delta_2}$. Finally, we have
\[
\begin{split}
\A^*(\Delta_1\hash\Delta_2)&=\A^*(V_{\Delta_1\hash\Delta_2})\cong
\A^*(V_{\Delta_1\hash\Delta_2\setminus I})\\
&\cong\A^*( V_{\Delta_1\setminus I}+V_{\Delta_2\setminus
I})\cong\A^*(\Delta_1)\hash\A^*(\Delta_2).
\end{split}
\]
\end{proof}

%
%

\subsection{Flips}\label{subsecFlips}

In this section we assume that $\Delta$ is based on an oriented
pseudomanifold $K$. Our goal is to define a flip in a multi-fan.
Consider separately two situations.

(1) Flips changing the number of vertices. Let us take a maximal
simplex $I\in K$, $|I|=n$. Let $\Flip^1_I(K)$ be a simplicial
complex whose maximal simplices are the same as in $K$ except that
we substitute $I$ by $\cone\dd I$. This operation adds the new
vertex $i$, the apex of the cone. If $\lambda\colon[m]\to \Ro^n$
is a characteristic function on $K$, we extend it to the set
$[m]\sqcup\{i\}$ by adding new value $\lambda(i)$ such that the
result is a characteristic function on $\Flip^1_I(K)$. This
defines an operation on multi-fans which we call \emph{the flip of
type $(1,n)$}.

The inverse operation will be denoted $\Flip^n_i$. It is
applicable to $\Delta$ if $\lk_Ki$ is isomorphic to the boundary
of a simplex and $\lambda(\ver(\lk_Ki))$ is a linearly independent
set. The inverse operation will be called \emph{the flip of type
$(n,1)$}.

(2) Flips preserving the set of vertices. Let $S$ be a subset of
$\ver(K)$ of cardinality $n+1$ such that the induced subcomplex
$K_S$ on the set $S$ is isomorphic to
$\dd\Delta^{p-1}\ast\Delta^{q-1}$ with $p+q=n+1$, $p,q\geqslant
2$. Let $\Flip^p_S(K)$ be the simplicial complex whose maximal
simplices are the same as in $K$ away from $S$, and
$\dd\Delta^{p-1}\ast\Delta^{q-1}$ is replaced by
$\Delta^{p-1}\ast\dd\Delta^{q-1}$. If the set of vectors
$\lambda(S)$ is in general position, then this operation is
defined on multi-fans. We call it \emph{the flip of type $(p,q)$}.
It is easily seen that flips of types $(p,q)$ and $(q,p)$ are
inverse to each other.

Of course $(1,n)$- and $(n,1)$-flips can be viewed as particular
cases of this construction if we allow ghost vertices and formally
set $\dd\Delta^0$ to be such a ghost vertex.

The following proposition tells that dimension vectors of
multi-fan algebras change under the flips in the same way as
$h$-vectors of simplicial complexes.

\begin{thm}\label{thmFlipChange}
Let $\Delta'$ be a multi-fan obtained from $\Delta$ by a
$(p,q)$-flip, $p,q\geqslant 1$. Then
\[
\dm(\Delta')-\dm(\Delta)=(\underbrace{1,\ldots,1}_{q},0,\ldots,0)-
(\underbrace{1,\ldots,1}_{p},0,\ldots,0).
\]
\end{thm}

\begin{proof}
For $(1,n)$- and $(n,1)$-flips this follows from Proposition
\ref{propConSum} and Lemma \ref{lemElementaryMultifan}, since
$(1,n)$-flip is just the connected sum with an elementary
multi-fan, and $(n,1)$-flip is its inverse.

Now we consider the remaining cases. $(p,q)$-flips with $p\neq 1$
and $q\neq 1$ do not change the vertex set. Let $[m]$ denote the
vertex set of $K$ and $K'$, and $S\subset [m]$ denote the set of
vertices at which the flip is performed. We have $|S|=n+1$ and
$K|_S\cong \dd\Delta^{p-1}\ast\Delta^{q-1}$ and $K'|_S\cong
\Delta^{p-1}\ast\dd\Delta^{q-1}$. Let $[p]$ be the set of vertices
of $\Delta^{p-1}$. Let $\I_{[m]\setminus S}$ denote the ideal in
$\Ro[\dd_1,\ldots,\dd_m]$ generated by $\dd_i$, $(i\in
[m]\setminus S)$.

\begin{claim}
$\I_{[m]\setminus S}\cap \Ann V_{\Delta}=\I_{[m]\setminus S}\cap
\Ann V_{\Delta'}$.
\end{claim}

\begin{proof}
In the group of multi-fans with a given characteristic function we
have a relation $\Delta'=\Delta+T$, where $T$ is an elementary
multi-fan based on the vertex set $S$. Informally, to perform a
flip on a multi-fan is the same as ``to add a boundary of a
simplex'', which cancels the cones from
$\dd\Delta^{p-1}\ast\Delta^{q-1}$ and adds the cones from
$\Delta^{p-1}\ast\dd\Delta^{q-1}$. Therefore
$V_{\Delta'}=V_{\Delta}+V_T$, where $V_T$ is the polynomial which
essentially depends only on the variables $c_i$, $i\in S$. If
$D\in \Ann V_{\Delta}\cap \I_{[m]\setminus S}$ then $D$
annihilates both $V_{\Delta}$ and $V_T$. Thus it annihilates
$V_{\Delta'}=V_\Delta+V_T$ and the claim follows.
\end{proof}

We have a diagram of inclusions of graded ideals in
$\Ro[\dd_1,\ldots,\dd_m]$:

\[
\xymatrix{\Ann V_{\Delta}\ar@{^{(}->}[d]&\I_{[m]\setminus S}\cap
\Ann V_{\Delta}=\I_{[m]\setminus S}\cap \Ann V_{\Delta'}
\ar@{_{(}->}[l]\ar@{^{(}->}[r]\ar@{^{(}->}[d]&\Ann V_{\Delta'}\ar@{^{(}->}[d]\\
\I_{[m]\setminus S}+\Ann V_{\Delta}&\I_{[m]\setminus
S}\ar@{_{(}->}[l]\ar@{^{(}->}[r]&\I_{[m]\setminus S}+\Ann
V_{\Delta'}}
\]

It follows that the quotients of the vertical inclusions are
isomorphic as graded vector spaces. Therefore
\begin{multline}
\dm(\Delta')-\dm(\Delta) = \dm(\Ro[m]/\Ann
V_{\Delta'})-\dm(\Ro[m]/\Ann V_{\Delta}) \\ = \dm
(\Ro[m]/(\I_{[m]\setminus S}+\Ann V_{\Delta'})) - \dm
(\Ro[m]/(\I_{[m]\setminus S}+\Ann V_{\Delta})).
\end{multline}
Since $\I_{[m]\setminus S}$ is the ideal generated by $\dd_i$,
$(i\notin S)$, the ring $\Ro[m]/(\I_{[m]\setminus S}+\Ann
V_{\Delta})$ coincides with some quotient ring $B$ of the
polynomials in variables $\dd_i$, $(i\in S)$, that is
$B=\Ro[S]/Rels$. The linear relations $\theta_j=\sum_{i\in[m]}
\lambda_{i,j}\dd_i$ in $\Ann V_\Delta$ induce the relations
$\sum_{i\in S} \lambda_{i,j}\dd_i$ in $Rels$. Since the values of
$\lambda$ on $S$ are in general position, these induced relations
are linearly independent. We have $n$ linear relations on $n+1$
variables, thus all variables are expressed through a single
variable $t$, and we have $B\cong \Ro[t]/\ca{J}$. Since we are in
a graded situation, and $B$ is a finite dimensional algebra,
$\ca{J}$ is a principal ideal generated by $t^{\tilde{p}}$ for
some $\tilde{p}\geqslant 0$. Hence
\[
\dm(\Ro[m]/(\I_{[m]\setminus S}+\Ann V_{\Delta})) = \dm B =
(\underbrace{1,\ldots,1}_{\tilde{p}},0,\ldots,0).
\]
Now notice that there is a Stanley--Reisner relation $\prod_{i\in
[p]}\dd_i=0$ corresponding to non-simplex $[p]$ in $K$ (recall
that $[p]$ is the set of vertices of $\dd\Delta^{p-1}$ inside
$\dd\Delta^{p-1}\ast\Delta^{q-1}\subset K$). Therefore we have a
relation $t^p=0$ in $B$. This implies $\tilde{p}\leqslant p$.

Applying the same arguments to $\Delta'$, we get
\[
\dm(\Ro[m]/(\I_{[m]\setminus S}+\Ann V_{\Delta'})) =
(\underbrace{1,\ldots,1}_{\tilde{q}},0,\ldots,0),
\]
where $\tilde{q}\leqslant q$. Hence we have
\[
\dm(\Delta')-\dm(\Delta)=(\underbrace{1,\ldots,1}_{\tilde{q}},0,\ldots,0)-
(\underbrace{1,\ldots,1}_{\tilde{p}},0,\ldots,0).
\]
Note that the vector on the left hand side is symmetric. Hence the
vector on the right hand side is symmetric. If at least one
inequality $\tilde{p}\leqslant p$ or $\tilde{q}\leqslant q$ is
strict, the vector at the right is not symmetric. Thus
$\tilde{p}=p$, $\tilde{q}=q$, and the statement follows.
\end{proof}

%
%
%
%
%
%

\section{Cohomology of torus manifolds}\label{secTorusManif}

\subsection{Multi-fans of torus manifolds}\label{subsecTorMfdsDefs}

Recall that \emph{a torus manifold} $X$ is an oriented closed
manifold of dimension $2n$ with an effective action of
$n$-dimensional compact torus $T$ having at least one fixed point,
and prescribed orientations of characteristic submanifolds. Any
torus manifold determines a non-singular multi-fan in the lie
algebra $L(T)\cong\Ro^n$ of the torus as follows (see details in
\cite{HM}).

Let $X_i$, $i\in[m]$ be \emph{the characteristic submanifolds}.
Let $M$ be a connected component of a non-empty intersection
$X_{i_1}\cap\cdots\cap X_{i_k}$ for some $k>0$ and
$\{i_1,\ldots,i_k\}\subset[m]$, and assume that $M$ has at least
one fixed point. Such submanifold will be called \emph{a face
submanifold}. We also assume that the manifold $X$ itself is a
face submanifold corresponding to $k=0$. It easily follows from
the transversality of characteristic submanifolds that $M$ has
codimension $2|k|$. Let $\Sigma_X$ be a poset of all face
submanifolds of $X$ ordered by reversed inclusion. The basic
representation theory of a torus implies that $\Sigma_X$ is a pure
simplicial poset of dimension $n-1$ on the vertex set $[m]$. The
maximal simplices of $\Sigma_X$ correspond to the fixed points of
$X$.

Given orientations of $X$ and $X_i$, $i\in[m]$, each fixed point
obtains a sign. This determines a sign function
$\sigma_X\colon\Sigma_X\indn\to\{-1,+1\}$.

Finally, let $T_i$ denote a circle subgroup fixing $X_i$, for
$i\in[m]$. The orientation of $X_i$ determines the orientation of
the $2$-dimensional normal bundle of $X_i$, which in turn
determines an orientation of $T_i$. Therefore we have a
well-defined primitive element
\[
\lambda_X(i)\in \Hom(S^1,T^n)\cong\Zo^n\subset \Ro^n\cong L(T^n).
\]
This gives a characteristic function $\lambda_X\colon[m]\to
\Ro^n$. These constructions determine a multi-fan
$\Delta_X:=(\Sigma_X,\sigma_X,\lambda_X)$ associated with a torus
manifold $X$. This multi-fan is non-singular and complete
\cite{HM}.

As described in subsection \ref{subsecMFansWeightCharFunc}, we may
turn the data ``simplicial poset + sign function'' into the data
``simplicial complex + weight function''. Let $K_X$ and $w_X$
denote the simplicial complex and the weight function
corresponding to $\Delta_X$.

In the following we assume that each non-empty intersection of
characteristic submanifolds is connected and contains at least one
fixed point. The assumption implies, in particular, that
$\Sigma_X$ is a simplicial complex, and therefore $K_X=\Sigma_X$
and the weight function $w_X$ coincides with $\sigma_X$.

\subsection{Face subalgebra in cohomology}\label{subsecFaceSubAlg}

Let $X$ be a torus manifold and $\Delta_X$ be the corresponding
multi-fan. Let $\F^*(X)\subset H^*(X;\Ro)$ be the vector subspace
spanned by the cohomology classes Poincare dual to face
submanifolds. Since the intersection of two face submanifolds is
either a face submanifold or empty, $\F^*(X)$ is a subalgebra.
This subalgebra is multiplicatively generated by the classes of
characteristic submanifolds.

The simplices of $K_X=\Sigma_X$ correspond to face submanifolds,
and there exists a ring homomorphism $\Ro[K_X]\to H^*_T(X)$
defined as follows: the element $x_I=\prod_{i\in I}x_i\in
\Ro[K_X]_{2|I|}$ corresponding to the simplex $I\in K_X$ maps to
the equivariant cohomology class dual to the face submanifold
$X_I=\bigcap_{i\in I}X_i$. There is a natural homomorphism
$H^*_T(X)\to H^*(X)$ induced by the inclusion of a fiber in the
Borel fibration $X\hookrightarrow X\times_T ET
\stackrel{\pi}{\rightarrow} BT$. We have a commutative square of
algebra homomorphisms
\begin{equation}\label{eqCohomSquareRings}
\begin{CD}
H_T^*(\Delta_X) @>>> H^*(\Delta_X)\\
@VVV @VVV \\
H^*_T(X) @>>> H^*(X).
\end{CD}
\end{equation}
Recall from the end of subsection \ref{subsecSRrings} that
$H_T^*(\Delta_X)$ denotes the Stanley--Reisner ring of the
underlying simplicial complex $K_X$, $H^*(\Delta_X)$ is its
quotient by the linear system of parameters, so the upper
horizontal arrow in the diagram \eqref{eqCohomSquareRings} is the
natural quotient homomorphism. By definition, $\F^*(X)$ is the
image of the right vertical map. Hence we have an epimorphism of
algebras $H^*(\Delta_X)\twoheadrightarrow \F^*(X)$.

\begin{thm}\label{thmFacePartCohom}
There exists a well-defined epimorphism of algebras
$\F^*(X)\twoheadrightarrow \A^*(\Delta_X)$.
\end{thm}

\begin{proof}
The epimorphism $H^*(\Delta_X)\to \F^*(X)$ is compatible with the
integration maps $\int_{\Delta_X}\colon H^{2n}(\Delta_X)\to \Ro$
(the multi-fan integration) and $\int_{X}\colon \F^{2n}(X)\to \Ro$
(integration over the manifold $X$), see \cite{HM}. Lemma
\ref{lemPDinduced} implies that the induced map
$\PD(H^*(\Delta_X),\int_{\Delta_X})\to \PD(\F^*(X),\int_{X})$ is
an isomorphism. Thus we have a natural epimorphism
$\F^*(X)\twoheadrightarrow \PD(\F^*(X),\int_{X})\cong
\PD(H^*(\Delta_X),\int_{\Delta_X})=\A^*(\Delta_X)$.
\end{proof}

Therefore the part of the cohomology ring generated by
characteristic submanifolds is clamped between two algebras
defined combinatorially:
\begin{equation}\label{eqCohomologyDiagr}
\xymatrix{ H^*(\Delta_X)\ar@{->>}[r]
& \F^*(X)\ar@{->>}[r] \ar@{^{(}->}[d] & \A^*(\Delta_X) \\
& H^*(X) & }
\end{equation}

\begin{cor}
Betti numbers of a torus manifold $X$ are bounded below by the
dimensions of graded components of $\A^*(\Delta_X)$.
\end{cor}

\begin{rem}
For complete smooth toric varieties and for quasitoric manifolds
all arrows in the diagram above are isomorphisms as follows from
Danilov--Jurkiewicz and Davis--Januszkiewicz \cite{DJ} theorems
respectively. If $K_X$ is a sphere, both horizontal maps are
isomorphisms, so the face part of cohomology is completely
determined by a multi-fan, while the vertical map may be
non-trivial. As an example in which such phenomena occur, take an
equivariant connected sum of a (quasi)toric manifold with a
manifold on which the torus acts freely and whose orbit space has
nontrivial homology. Finally, there exist many examples of torus
manifolds for which all arrows in \eqref{eqCohomologyDiagr} are
nontrivial \cite{AyV3}.
\end{rem}

Recall that $H^*(\Delta_X)$ is linearly generated by the
square-free monomials $x_I$ corresponding to simplices $I\in K$.
There may exist certain linear relations on these elements coming
from Minkowski relations: $\sum_{I:|I|=k}a(I)\lambda(I)_\mu x_I$,
where $\mu\in \Lambda^kV$ and $a$ is a function on
$(k-1)$-simplices of $K$.

\begin{conj}
As a vector space, $\F^{2k}(X)$ is generated by the elements
$\{x_I\}_{I\in K}$ subject to the Minkowski relations
$\sum_{I:|I|=k}a(I)\lambda(I)_\mu x_I=0$, where $\mu$ runs over
$\Lambda^kV$ and $a$ runs over all functions such that the element
\[
\sum_{I:|I|=k}a(I)[X_I/T]
\]
bounds in $C_{n-k}(X/T;\Ro)$.
\end{conj}

This question is closely related to Problem
\ref{problMinkowskiRels}. It can be seen that whenever
$\sum_{I:|I|=k}a(I)[X_I/T]$ bounds in $C_{n-k}(X/T;\Ro)$, the
element $\sum_{I:|I|=k}a(I)F_I$ is a cycle of multi-polytopes,
therefore $\sum_{I:|I|=k}a(I)\lambda(I)_\mu x_I$ vanishes in
$\A^*(\Delta_X)$. However, there may be cycles of multi-polytopes
such that the corresponding elements $\sum_{I:|I|=k}a(I)[X_I/T]$
do not bound in the orbit space. This observation represents the
fact that the right arrow in \eqref{eqCohomologyDiagr} can be
nontrivial.

If $X$ is an oriented manifold with locally standard torus action,
having trivial free part and acyclic proper faces of the orbit
space, the conjecture is proved in \cite{AyV3}. Informally, this
situation corresponds to the case when the underlying simplicial
complex of $\Delta_X$ is an oriented homology manifold.

\end{document}